\documentclass[a4paper,11pt]{article}
\usepackage[latin1]{inputenc}
\usepackage{amsfonts}
\usepackage{amsthm}
\usepackage{amsmath}
\usepackage{amsfonts}
\usepackage{latexsym}
\usepackage{amssymb}
\usepackage{dsfont}

\newtheorem{fed}{Definition}[section]
\newtheorem{teo}[fed]{Theorem}
\newtheorem*{teo*}{Theorem}
\newtheorem{lem}[fed]{Lemma}
\newtheorem{cor}[fed]{Corollary}
\newtheorem{pro}[fed]{Proposition}
\theoremstyle{definition}
\newtheorem{rem}[fed]{Remark}
\newtheorem{conj}[fed]{Conjecture}

\newtheorem{exa}[fed]{Example}

\newtheorem{num}[fed]{}
\def\ID{Let $(\cF_0 \coma \ca)$ be initial data
for the CP} 
\def\ga{\gamma}

\def\n0{n_{ \text{\rm \tiny o}}}

\newcommand{\IN}[1]{\mathbb {I} _{#1}}
\def\In{\mathbb {I} _n}
\def\IM{\mathbb {I} _m}
\def\suml{\sum\limits}

\def\bce{\begin{center}}
\def\ece{\end{center}}
\newcommand{\trivial}{\{0\}}
\DeclareMathOperator{\FP}{FP\,}

\def\subim{_{i\in \IN{n}}\,}
\def\py{\peso{and}}
\def\rk{\text{\rm rk}}
\def\noi{\noindent}
\def\cF{\mathcal F}
\def\cG{\mathcal G}
\def\QED{\hfill $\square$}
\def\EOE{\hfill $\triangle$}
\def\EOEP{\tag*{\EOE}}
\newcommand{\peso}[1]{ \quad \text{ #1 } \quad }
\def\uno{\mathds{1}}
\def\bm{\left[\begin{array}}
\def\em{\end{array}\right]}
\def\ben{\begin{enumerate}}
\def\een{\end{enumerate}}
\def\bit{\begin{itemize}}
\def\eit{\end{itemize}}
\def\barr{\begin{array}}
\def\earr{\end{array}}
\def\igdef{\ \stackrel{\mbox{\tiny{def}}}{=}\ }

\def\k{n}

\def\eps{\varepsilon}
\def\la{\lambda}
\def\al{\alpha}
\def\N{\mathbb{N}}
\def\R{\mathbb{R}}
\def\C{\mathbb{C}}

\def\cC{\mathcal{C}}

\def\cH{\mathcal{H}}
\def\cK{\mathcal{K}}

\def\cR{{\cal R}}
\def\cS{{\cal S}}

\def\cM{{\cal M}}
\def\cB{{\cal B}}

\def\cV{{\cal F}}
\def\cU{{\cal U}}
\def\cW{{\cal G}}
\def\ca{\mathbf{a}}
\def\cb{\mathbf{a}}

\def\inc{\subseteq}

\def\sii{ if and only if }
\def\inv{^{-1}}

\def\api{\langle}
\def\cpi{\rangle}

\def\efin{E(\cF_0 \coma \cb)}
\def\efinb{E(\cF_0 \coma \cb)}

\def\poRt{L_d(\cR)^+_\tau}
\def\poRta{L(\cR)^+_\tau}
\def\stau{S_\tau}

\def\nfin{N(\cF_0 \coma \cb)}
\def\ua{^\uparrow}
\def\da{^\downarrow}
\def\caop{\cC_\ca^{\rm op}(\cF_0)}
\def\s{\mathfrak{s}}

 \DeclareMathOperator{\tr}{tr}
\DeclareMathOperator{\gen}{span}

\DeclareMathOperator{\leqp}{\leqslant}

\def\RS{\mathbf{F} }

\def\RSV{\cF= \{f_i\}_{i\in \, \IN{n}}}

\def\coma{\, , \, }

\newcommand{\hil}{\mathcal{H}}
\newcommand{\op}{L(\mathcal{H})}
\newcommand{\lhk}{L(\mathcal{H} \coma \mathcal{K})}

\newcommand{\lkh}{L(\mathcal{K} \coma \mathcal{H})}

\newcommand{\posop}{L(\mathcal{H})^+}

\def\H{{\cal H}}
\def\glh{\mathcal{G}\textit{l}\,(\cH)}

\newcommand{\mat}{\mathcal{M}_d(\mathbb{C})}

\newcommand{\matsa}{\mathcal{H}(n)}
\newcommand{\matsad}{\mathcal{H}(d)}

\newcommand{\matu}{\mathcal{U}(n)}
\newcommand{\matud}{\mathcal{U}(d)}

\newcommand{\matpos}{\mat^+}

\newcommand{\matinvd}{\mathcal{G}\textit{l}\,(d)}
\def\gld{\matinvd^+}

\newcommand{\matrec}[1]{\mathcal{M}_{#1} (\mathbb{C})}

\def\beq{\begin{equation}}
\def\eeq{\end{equation}}

\def\pausa{\medskip\noi}

\begin{document}

\title{ {\bf Optimal completions of a frame.}}
\author{Pedro G. Massey, Mariano A. Ruiz  and Demetrio Stojanoff\thanks{Partially supported by CONICET 
(PIP 5272/05) and  Universidad Nacional de La PLata (UNLP 11 X472).} }
\author{P. G. Massey, M. A. Ruiz and D. Stojanoff \\ {\small Depto. de Matem\'atica, FCE-UNLP,  La Plata, Argentina
and IAM-CONICET  \footnote{e-mail addresses: massey@mate.unlp.edu.ar , mruiz@mate.unlp.edu.ar , demetrio@mate.unlp.edu.ar}
}}
\date{}
\maketitle

\begin{abstract}  Given a finite sequence of vectors $\mathcal F_0$ in $\C^d$ we describe the spectral and geometrical structure of optimal completions of $\mathcal F_0$ obtained by adding a finite sequence of vectors with prescribed norms, where optimality is measured with respect to a general convex potential. In particular, our analysis includes the so-called Mean Square Error (MSE) and the Benedetto-Fickus' frame potential. On a first step, we reduce the problem of finding the optimal completions to the computation of the minimum of a convex function in a convex compact polytope in $\R^d$. As a second step, we show that there exists a finite set (that can be explicitly computed in terms of a finite step algorithm that depends on $\cF_0$ and the sequence of prescribed norms) such that the optimal frame completions with respect to a given convex potential can be described in terms of a distinguished element of this set. As a byproduct we characterize the cases of equality
in Lindskii's inequality from matrix theory. 
\end{abstract}

\noindent  AMS subject classification: 42C15, 15A60.

\noindent Keywords: frames, frame completions, majorization, Lindskii's inequality,
Schur-Horn theorem.

\date{}

\tableofcontents

\section{Introduction} 

A finite sequence of vectors $\cF=\{f_i\}_{i\in \IN{n}}$ in $\C^d$ is a frame for $\C^d$ if the sequence spans $\C^d$. It is well known that finite frames provide redundant linear encoding-decoding schemes, that have proved useful in real life applications. Conversely, several research problems in this field have arise in the attempt to apply this theory in different contexts.

 For example, the (linear) redundancy provided by finite frames translates into robustness properties of the transmission scheme that they induce, which make frames a useful device for transmission of signals through noisy channels; this last fact has posed several problems dealing with the determination of what is known in the literature as optimal frames for erasures (see \cite{Bod,Pau,BodPau,caskov,HolPau,LoHanagre,LeHanagre}). 
 
 On the other hand, the so-called tight frames allow for redundant linear representations of vectors (signals) that are formally analogous to the linear representations given by orthonormal basis; this feature makes tight frames a distinguished class of frames that is of interest for applications. Conversely, in several applications we would like to consider tight frames that have some other prescribed properties leading to what is known in the literature as  frame design problems \cite{Illi,CFMP,Casagregado,CMTL,ID,DFKLOW,FWW,KLagregado}. It is worth pointing out that in some cases it is not possible to find a frame fulfilling the previous demands; in \cite{BF}
Benedetto and Fickus found an alternative approach to these situations 
by introducing a functional, called the frame potential, and showing that minimizers of the frame potential (within a convenient set of frames) are the natural substitutes of tight frames with prescribed parameters (see also \cite{Phys,FJKO,JOk,MR} and \cite{casafick3,MRS,MRS2} for related problems in the context of fusion frames).

Recently, the following frame completion problem, related with the frame design problems mentioned above, was posed in \cite{FicMixPot}: given
an initial sequence $\cF_0$ in $\C^d$ and a sequence of positive numbers $\cb$ then compute the sequences $\cG $ in $\C^d$ whose elements have norms given by the sequence $\cb$ and such that 
the completed sequence $\cF=(\cF_0 \coma \cG )$ minimizes the so-called mean square error (MSE) of $\cF$, which is a (convex) functional (see also \cite{CCHK, FWW,MR0} for completion problems for  frames). The initial sequence of vectors can be considered as a checking device for the measurement, and therefore we search for a complementary set of measurements (given by vectors with prescribed norms) in such a way that the complete set of measurements is optimal with respect to the MSE. Notice there are other possible (convex) functionals that we could choose to minimize such as, for example, the frame potential. Therefore, a natural extension of the previous problem is: given a (convex) functional defined on the set of frames, compute the frame completions with prescribed norms that minimize this functional. 

A first step towards the solution of this general version 
of the completion problem was made in \cite{MRS4}. There we showed 
that under certain hypothesis (feasible cases, see Section \ref{problemon}), optimal frame completions 
with prescribed norms do not depend on the particular choice of convex 
functional, as long as we consider 
{\it convex potentials}, that contain the MSE and the frame potential (see Section \ref{basic}). 
On the other hand, it is easy to show examples in which the previous 
result does not apply (non-feasible cases); in these cases the optimal 
frame completions with prescribed norms are not known even for the 
MSE nor the frame potential.

In this paper we consider the frame completion problem of an initial 
sequence $\cF_0$ in $\C^d$, for general sequences $\cb$ of prescribed norms 
and for a fixed convex potential $P_f$ - where $f$ is a strictly convex function - 
in the non-feasible cases (see Section \ref{problemon} for motivations 
and a detailed description of our main problem). In order to deal with the 
general problem we introduce and develop a class of pairs of positive 
matrices (called optimal matchings matrices, see the Appendix) that allow 
to reduce the problem to the computation of minimizers of a scalar convex function $F$ 
(associated to $f$) in a compact convex domain in $\R^d$ (the same 
set for every map $f$). This constitutes a reduction of the optimization 
problem, that in turn can be attacked with several numerical tools in concrete examples.
In fact, the convex domain has a natural and explicit description in terms 
of majorization, which is an algorithmic notion.

We also study the spectral and geometrical structure of local minimizers of 
$P_f$ in the set of frame completions with prescribed norms, in terms of a 
geometrical approach to a perturbation problem. These last results allow to 
a second reduction of the problem: there is a finite set 
$\efinb $ in $\R^d$ - that depends only on the initial family $\cF_0$ 
and the finite sequence $\cb$ of positive numbers - such that for any fixed 
convex potential $P_f$ there exists a unique 
vector $\mu = \mu_f \in  \efinb $ (computable by a minimization on the finite set $\efinb$ 
in terms of $F$) such that all optimal frame completions for $P_f$ with 
prescribed norms can be computed in terms of $\mu$.

In both methods, we describe the optimal vector of eigenvalues for the 
frame operator of the completing sequences. With this data, the optimal 
completions (which satisfy the norm restrictions) can be effectively computed by using a well known 
algorithm developed in \cite{ID} that implements  the Schur-Horn theorem.

 In all  examples that we have computed numerically, we have found that the optimal spectrum of the completing sequences does not depend on the particular choice of convex potential $P_f$ considered. Although at the present we have not been able to prove this fact, we state it as a conjecture. We have also observed two other common features of optimal solutions - that are also stated as conjectures - that allow to implement an efficient (and considerably faster) algorithm that computes an smaller set than $\efinb$ that also enables to compute the optimal frame completions with prescribed norms with respect to a general convex potential $P_f$.

The paper is organized as follows. In Section \ref{Pre} we state several facts and notions about frame theory in finite dimension and majorization, which is a notion from matrix theory; in this section we describe in detail the main problem of the present paper and some previous related results. In section \ref{sec3} we reduce the problem of computing optimal frame completions with prescribed norms to a set of completions whose frame operators are optimal matchings of the frame operator of the initial set of vectors $\cF_0$, in the sense described in the Section \ref{secAppendix} (Appendix). 
Based on the results of the Appendix we obtain a first reduction of the problem and show that the optimal frame completions with prescribed norms for the convex potential $P_f$ can be described in terms of the minimizers of an associated function $F$ in a compact convex polytope. We also show that the spectral structure of optimal completions is unique and has some other features. In Section \ref{Sec4} we introduce two different topologies in the set of completions and consider the geometrical structure of local minimizers with respect to these topologies; in order to do this we apply tools from differential geometry that allow to solve a local perturbation problem for frames with prescribed norms. 
Using these results we show in Section \ref{sec5} that optimal completions $\cF=(\cF_0,\cG)$ have the property that the vectors of the completing sequence $\cG$ are eigenvectors of the frame operator $S_\cF$ of the complete sequence $\cF$. Based on this last fact we develop an algorithm (that can be effectively implemented) to compute optimal completions numerically. The analysis of the computed examples reflects some commons features of the numerical solutions. Based on these facts we state some other conjectures related with the spectral structure of optimal completions. Finally, in Section \ref{secAppendix} we introduce pairs of positive matrices, that we call optimal matchings, and describe the structure of these pairs; this corresponds to the study of the case of equality in Lindskii's inequality from matrix theory.

\section{Preliminaries}\label{Pre}

In this section we describe the basic notions that we shall consider throughout the paper. We first establish the general notations and then we recall the basic facts from frame theory that are related with our main results. Then, we  describe submajorization which is a notion from matrix analysis, that will play a major role in this note. Finally, we recall the solution of the frame design problem in terms of majorization and give a detailed description of the optimal frame completion problem, which is the main topic of this paper. 

\subsection{General notations.}
Given $m \in \N$ we denote by $\IM = \{1, \dots , m\} \inc \N$ and 
$\uno = \uno_m  \in \R^m$ denotes the vector with all its entries equal to $1$. 
For a vector $x\in \R^m$ we denote by $x^\downarrow$ (resp. $x^\uparrow$) the rearrangement
of $x$ in  decreasing (resp. increasing) order, and $(\R^m)^\downarrow = \{ x\in \R^m : x = x^\downarrow\}$ the set of downwards ordered vectors. 

\pausa
Given $\cH \cong \C^d$  and $\cK \cong \C^n$, we denote by $\lhk $ 
the space of linear operators $T : \cH \to \cK$. 
Given an operator $T \in \lhk$, $R(T) \inc \cK$ denotes the
image of $T$, $\ker T\inc \cH$ the null space of $T$ and $T^*\in \lkh$ 
the adjoint of $T$. 
If $\cK = \cH$ we denote by $\op = L(\cH \coma \cH)$, 
by $\glh$ the group of all invertible operators in $\op$, 
 by $\posop $ the cone of positive operators and by
$\glh^+ = \glh \cap \posop$. 
If $T\in \op$, we  denote by   
$\sigma (T)$ the spectrum of $T$, by $\rk\, T= \dim R(T) $  the rank of $T$,
and by $\tr T$ the trace of $T$. 

\pausa 
By fixing orthonormal basis's (ONB's) 
of the Hilbert spaces involved, we shall identify operators with 
matrices, using the following notations: 
by $\matrec{n,d} \cong L(\C^d \coma \C^n)$ we denote the space of complex $n\times d$ matrices. 
If $n=d$ we write $\mat = \matrec{d,d}$ ;   
$\matsad$ is the $\R$-subspace of selfadjoint matrices,  
$\matinvd$ the group of all invertible elements of $\mat$, $\matud$ the group 
of unitary matrices, 
$\matpos$ the set of positive semidefinite
matrices, and $\matinvd^+ = \matpos \cap \matinvd$.

\pausa
If $W\inc \cH$ is a subspace we denote by $P_W \in \posop$ the orthogonal 
projection onto $W$, i.e. $R(P_W) = W$ and $\ker \, P_W = W^\perp$. 
Given $x\coma y \in \cH$ we denote by $x\otimes y \in \op$ the rank one 
operator given by 
$x\otimes y \, (z) = \api z\coma y\cpi \, x$ for every $z\in \cH$. Note that
if $\|x\|=1$ then $x\otimes x = P_{\gen\{x\}}\,$.

\pausa
Given $S\in \matpos$, we write $\la(S) \in 
(\R_{\geq 0}^d)^\downarrow$ the 
vector of eigenvalues of $S$ - counting multiplicities - arranged in decreasing order. 
If $ \lambda(S)=\la= (\lambda_1 \coma \ldots \coma \lambda_d) \in (\R_{\geq 0}^d)^\downarrow\,$, a system 
$\cB=\{h_i\}_{i\in \IN{d}} \inc \C^d$ is a 
``ONB of eigenvectors for $S\coma \la \,$" if it is an  
orthonormal basis for $\C^d$ such that 
$S\,h_i=\lambda_i\,h_i$ for every $i\in \IN{d}\,$. 
In other words, an  orthonormal basis 
\beq\label{BON S la}
\mbox{$\cB=\{h_i\}_{i\in \IN{d}} $ \ \ \ is a 
``ONB of eigenvectors for $S\coma \la \,$"} \iff 
S = \sum_{i\in \IN{d}} \, \la_i \cdot h_i \otimes h_i \ .
\eeq
For vectors in $\C^d$ we shall use the euclidean norm.  
On the other hand,  for  $T\in \matrec{n\coma d}$ we shall use the spectral norm, denoted $\|T\|$,  given by $\|T\| =  \max\limits_{\|x\|=1}\|Tx\|$.

\subsection{Basic framework of finite frames }\label{basic}

In what follows we consider $(\k,d)$-frames. See 
\cite{BF,TaF,Chr,HLmem,MR} for detailed expositions of several aspects of this notion. 

\pausa
Let $d, n \in \N$, with $d\le n$. Fix a Hilbert space $\hil\cong \C^d$. 
A family $ \RSV \in  \cH^n $  is an 
$(\k,d)$-frame for $\cH$  if there exist constants $A,B>0$ such that
\beq\label{frame defi} A\|x\|^2\leq \sum_{i=1}^n |\left \langle x \, , f_i\right \rangle|^2\leq B \|x\|^2 \peso{for every} x\in \hil \ .
\eeq
The {\bf frame bounds}, denoted by  $A_\cF, B_\cF$ are the optimal constants in \eqref{frame defi}. If $A_\cF=B_\cF$ we call $\cF$ a tight frame.
Since $\dim \hil<\infty$, a family  $\RSV$  is an 
$(\k,d)$-frame 
 \sii $\gen\{f_i: i \in \In \} = \cH$.  
We shall denote by $\RS = \RS(n \coma d)$ the set of all $(\k,d)$-frames for $\cH$. 

\pausa
Given $\RSV \in  \cH^n $, the operator $T_\cV \in L(\hil\coma\C^n)$ defined
by 
 \beq \ T_\cV\, x= \big( \,\api x \coma f_i\cpi\,\big)  \subim 
\ , \peso{for every} x\in \cH \,
\eeq
is the {\bf analysis} operator of $\cF$.  Its adjoint $T_\cV^*$ is called the {\bf synthesis} operator: 
$$
T_\cV^* \in L(\C^n\coma \cH)  \peso{given by}
T_\cV ^* \, v =\sum_{i\in \, \IN{m}} v_i\, f_i \peso{for every}
v = (v_1\coma \dots\coma v_n)\in \C^n \ . 
$$
Finally, we define the {\bf frame operator} of $\cV$ as 
$$
S_\cV = T_\cV^*\  T_\cV = \sum_{i \in \In} f_i \otimes f_i 
\in \posop\ . 
$$
Notice that, if 
$\RSV \in  \cH^n $ then  
$\api S_\cV \, x\coma x\cpi \, = \sum\subim \, 
 \big|\, \api x \coma f_i\cpi \, \big|^2$ for every 
 $x\in \cH$. Hence, $\cF\in \RS(n \coma d)$ if and only if 
$S_\cF\in \glh^+$ and in this case
$A_\cV \, \|x\|^2 \, \le \, \api S_\cV \, x\coma x\cpi 
 \,  \le \, B_\cV \, \|x\|^2$ for every $ x\in \cH $.   
In particular, $A_\cV   =\la_{\min} (S_\cV) = \|S_\cV\inv \| \inv$ and $ 
\la_{\max} (S_\cV) = \|S_\cV \| = B_\cV \,$.
Moreover, $\cV $ is tight if and only if $S_\cV = 
\frac{\tau}{d}  \, I_\H\,$, where $\tau = 
\tr S_\cV = \sum\subim \|f_i\|^2 \,$.

\pausa
The frame operator plays an important role in the reconstruction of a vector $x$ using its frame coefficients 
$\{\api  x\coma f_i\cpi \,\}_{i\in \IN{n}}$. This leads to the definition of the canonical dual frame associated to $\cF$:
for every $\RSV\in \RS(n \coma d)$, the {\bf canonical dual} frame associated to $\cV$ is the sequence 
$\cV^\#\in \RS$ defined by
$$
\cV^\# \igdef S_\cV ^{-1} \cdot \cV  =  \{S_\cV ^{-1} \,f_i\,\}_{i\in \, \IN{m}} \in \RS(n \coma d) \ .
$$
Therefore, we obtain the reconstruction formulas
\begin{equation}\label{ec recons}
x= \sum\subim  \api  x \coma f_i\cpi \, S_\cV^{-1} \,f_i 
= \sum\subim  \api  x \coma S_\cV^{-1} \, f_i\cpi \, f_i  \peso{for every} x\in \cH \ .
\end{equation} 
Observe that the canonical dual $\cV^\#$ satisfies that given $x\in \cH$, then 
\beq\label{SMP}
T_{\cV^\#}\, x= \big( \,\api x \coma  S_\cV\inv \, f_i\cpi\,\big)  \subim 
= \big( \,\api S_\cV\inv \, x \coma   f_i\cpi\,\big)  \subim 
\peso{for} x\in \cH 
\implies 
T_{\cV^\#}= T_\cV \, S_\cV\inv \ .
\eeq 
Hence $T_{\cV^\#}^* \, T_\cV = 
I_\cH $ and $S_{\cV^\#} =  S_{\cV}^{-1}\,T_\cV^*\  T_\cV\,S_{\cV}^{-1} = S_{\cV}^{-1}\,$.  

\pausa 
In their seminal work \cite{BF}, Benedetto and Fickus introduced a functional defined (on unit norm frames), the so-called frame potential, given by $$ \FP(\{f_i\}_{i\in \In} )=\sum_{i,\,j\,\in \In}|
\api f_i\coma f_j \cpi |\,^2\ .$$ One of their major results shows that tight unit norm frames - which form an important class of frames because of their simple reconstruction formulas - can be characterized as (local) minimizers of this functional among unit norm frames. Since then, there has been interest in (local) minimizers of the frame potential within certain classes of frames, since such minimizers can be considered as natural substitutes of tight frames (see \cite{Phys,MR,MRS}). Notice that, given $\cF=\{f_i\}_{i\in \IN{n}}\in  \cH^n$ then $\FP(\cF)=\tr \, S_\cF^2 
=\sum_{i\in \IN{d}}\lambda_i(S_\cF)^2$.
These remarks have motivated 
the definition of general convex potentials as follows:

\begin{fed}\label{pot generales}\rm
Let $f:[0 \coma \infty)\rightarrow [0 \coma \infty)$ be a 
convex function. Following \cite{MR} we consider the 
(generalized) convex potential associated to $f$, denoted $P_f$, given by
\begin{equation}
P_f(\cF)=\tr \, f(S_\cF)  \peso {for} 
\cF=\{f_i\}_{i\in \IN{n}}\in  \cH^n \ .\EOEP
\end{equation}  
\end{fed}

\pausa
Of course, one of the most important convex potential is the Benedetto-Fickus' (BF) frame potential. As shown in \cite[Sec. 4]{MR} these convex functionals (which are related with the so-called entropic measures of frames) share many properties with the BF-frame potential. Indeed, under certain restrictions both the spectral and geometric structures of minimizers of these potentials coincide (see \cite{MR}).

\begin{rem} 
The results that we shall develop in this work apply in the case of convex potentials $P_f$ for a strictly convex function $f:[0,\infty)\rightarrow \R$.
Notice that this formulation does not formally include the Mean Square Error (MSE), which is the convex potential associated with the strictly convex function
$f:(0,\infty)\rightarrow (0,\infty)$ given by $f(x)= x\inv$, since $f$ is not defined in $0$ in this case. In order to include the MSE within our results we proceed as follows: 
we define $\tilde f:[0,\infty)\rightarrow (0,\infty]$ given by $\tilde f(x)=x^{-1}$ for $x>0$ and $\tilde f(0)=\infty$. 
Assuming that  $ x<\infty $ and $x+\infty =x\cdot \infty =\infty$ for 
every $x\in (0\coma \infty)$, it turns out that the new map $\tilde f$ is a (extended) strictly 
convex function and  all the results obtained in this paper apply to the convex 
potential induced by $\tilde f$.\EOE
\end{rem}

\subsection{Submajorization} \label{subsec 2.2}

Next we briefly describe submajorization, a notion from matrix analysis theory that will be used throughout the paper. For a detailed exposition of submajorization see \cite{Bat}.

\pausa 
 Given $x,\,y\in \R^d$ we say that $x$ is
{\bf submajorized} by $y$, and write $x\prec_w y$,  if
$$
\suml_{i=1}^k x^\downarrow _i\leq \suml_{i=1}^k y^\downarrow _i \peso{for every} k\in \mathbb I_d \ .
$$  
If $x\prec_w y$ and $\tr x = \sum_{i=1}^dx_i=\sum_{i=1}^d y_i = \tr y$,  then we say that $x$ is
{\bf majorized} by $y$, and write $x\prec y$. 
If the two vectors $x$ and $y$ have different size, we write 
$x\prec y$ if the extended vectors (completing with zeros 
to have the same size) satisfy the previous relationship.  

\pausa
On the other hand we write 
$x \leqp y$ if $x_i \le y_i$ for every $i\in \mathbb I_d \,$.  It is a standard  exercise 
to show that $x\leqp y \implies x^\downarrow\leqp y^\downarrow  \implies x\prec_w y $. 
Majorization is usually considered because of its relation with tracial inequalities 
for convex functions. 
Indeed, given $x,\,y\in \R^d$ and  $f:I\rightarrow \R$ a 
convex function defined on an interval $I\inc \R$ such that 
$x,\,y\in I^d$,  then (see for example \cite{Bat}): 
\ben 
\item If one assumes that $x\prec y$, then 
$ 
\tr f(x) \igdef\suml_{i=1}^df(x_i)\leq \suml_{i=1}^df(y_i)=\tr f(y)\ .
$
\item If only $x\prec_w y$,  but the map $f$ is also increasing, then  still 
$\tr f(x) \le \tr f(y)$. 
\item If $x\prec_w y$ and $f$ is a strictly convex function such 
that $\tr \,f(x) =\tr \, f(y)$ then there exists a permutation $\sigma$ 
of $\IN{d}$ such that $y_i=x_{\sigma(i)}$ for $i\in \IN{d}\,$. 
\een

\begin{rem}\label{rem doublystohasctic}
Majorization between vectors in $\R^d$ is intimately related with the 
class of doubly stochastic $d\times d$ matrices, denoted by DS$(d)$. Recall that a 
$d\times d$ matrix $D \in $ DS$(d)$   if it has non-negative entries and each row sum and column sum equals 1. 

\pausa
It is well known (see \cite{Bat}) that given $x\coma y\in \R^d$ then $x\prec y$ 
if and only if there exists $D\in$ DS$(d)$ such that $D\, y=x$. As a consequence 
of this fact we see that if $x_1\coma y_1\in \R^r$ and $x_2 \coma y_2\in \R^s$ 
are such that $x_i\prec y_i\, $, $i=1\coma 2$, 
then $x=(x_1\coma x_2)\prec y=(y_1\coma y_2)$ in $\R^{r+s}$.

\pausa
Indeed, if $D_1$ and $D_2$ are the doubly stochastic matrices corresponding the previous majorization relations then $D=D_1\oplus D_2\in$ DS$(r+s)$ is such that $D\, y=x$. 
\EOE
\end{rem}

\pausa
Submajorization can be extended to the context of self-adjoint matrices 
as follows: given $S_1\coma S_2\in \mathcal H(d)$ we say that $S_1$ 
is {\bf submajorized} by $S_2\, $, denoted $S_1\prec_w S_2\, $,  if 
$\lambda(S_1)\prec_w \lambda(S_2)\, $. If $S_1\prec_w S_2$ 
and $\tr \,  S_1=\tr \, S_2$ we say that $S_1$ is {\bf majorized} by $S_2$ and 
write $S_1\prec S_2\, $. Thus, $S_1\prec S_2$ if and only if 
$\lambda(S_1)\prec \lambda(S_2)$. Notice that (sub)majorization is 
a spectral relation between self-adjoint operators.

\pausa
We end this section by recalling the following result, known as Lindskii's inequality (see \cite[III.4]{Bat}). 
\begin{teo}[Lindskii's inequality ]\label{LindsTeo} 
Let $A,\,B\in \matsad$. Then $\lambda(A)+\lambda^\uparrow(B)\prec \lambda(A+B)$.
\QED
\end{teo}

\pausa
Lindskii's inequality plays an important role in our study of optimal frame completion problems. Moreover, the case of equality in Lindskii's inequality, i.e. when $(\lambda(A)+\lambda^\uparrow(B))^\downarrow =\lambda(A+B)$ for $A,\,B\in \matsad$, plays a central role in this paper. We completely characterize such pair of matrices - that we call optimal matching matrices - in the Appendix.

\subsection{Frames and optimal completions with prescribed parameters}\label{problemon}

In several applied situations it is desired to construct 
a sequence $\cF$ in such a way that the  frame 
operator of $\cF$ is given by some $S\in\matpos$ and the squared norms of the frame elements 
are prescribed by a sequence of positive numbers  $\ca=(\alpha_i)_{i\in \IN{n}}\in \R_{>0}^n\,$. 
That is, given a fixed $S\in \matpos$ and $\mathbf{a}\in \R_{>0}^n\,$, we analyze the existence (and construction) of a sequence $\cF=\{f_i\}_{i\in \IN{n}}$ such that $S_\cF=S$ 
and $\|f_i\|^2=\alpha_i\,$, for $i\in \IN{n}\,$. This is known as the classical 
frame design problem. It  has been treated by several research groups (see for example \cite{Illi,CFMP,Casagregado,CMTL,ID,DFKLOW,FWW,KLagregado}).
In what follows we recall 
a solution of the classical frame design problem in the finite dimensional setting, in the way that it is convenient for our analysis.

\begin{pro}[\cite{Illi,MR0}]\label{frame mayo}\rm
Let $B\in \matpos$ with $\lambda(B)\in \R_{+}^d\,^\downarrow$ and let 
$\cb=(\alpha_i)_{i\in\IN{k}} \in \R_{>0}^k\,$. Then there exists 
a sequence $\cW=\{g_i\}_{i\in\IN{k}}\in \cH^k$ with frame operator 
$S_\cW= B$ and such that $\|g_i\|^2=\alpha_i$ for every $i\in\IN{k}\,$ 
if and only if 
$\cb\prec \lambda(B)$ (completing with zeros if $k\neq d$).
\QED
\end{pro}

\pausa
Recently, researchers have made a step forward in the classical frame design problem and have asked about the structure of {\bf optimal} frames with prescribed parameters.
For example, consider the following problem posed in \cite{FicMixPot}:
let $\cH \cong \C^d$ and let $\cF_0=\{f_i\}_{i\in \IN{\n0}}
\in \cH^{\n0  }$ be a fixed (finite) sequence of vectors. 
Consider a sequence $\ca
= (\alpha_i)_{i\in \IN{k}} \in \R_{>0}^k\,$ such
that $\rk \, S_{\cF_0} \ge d-k$ and denote by $n=\n0+k$.  Then, 
with this fixed data, the problem is to construct a sequence  
$$
\cG = \{f_i\}_{i=\n0  +1}^n\in \cH^{k}  \peso{with} \|f_{\n0+i}\|^2=\alpha_i \peso{for}
1\leq i\leq k  \ , 
$$
such that the resulting completed sequence is a frame 
$\cF= (\cF_0\coma \cG )= \{f_i\}_{i\in \IN{n}}\in\RS(n \coma d)$  
whose MSE  $\tr \, S_\cF^{-1}$ is minimal among 
all possible such completions.

\pausa
Note that there are other possible ways to measure robustness 
(optimality) of the completed frame $\cF$ as above. 
For example, we can consider optimal (minimizing) completions, with prescribed norms, 
for the Benedetto-Fickus' potential. In this case we search for a frame 
$\cF= (\cF_0\coma \cG )= \{f_i\}_{i\in \IN{n}}\in\RS(n \coma d)$, with
$\|f_{\n0+i}\|^2=\alpha_i$ for $1\leq i\leq k$, and such that its frame potential 
$\FP(\cF)=\tr \,S_\cF ^2 $ is minimal  
among all possible such completions. Indeed, this problem has been 
considered before in the particular case in which $\cF_0=\emptyset$ 
in \cite{BF,Phys,FJKO,JOk,MR}.

\pausa
In this paper we shall consider the problems of optimal completion with prescribed norms, where optimality is measured with respect to general convex potentials (see Definition \ref{pot generales}). In order to describe our main problem we first fix the 
notation that we shall use throughout the paper.

\begin{fed}\rm 
Let $\cF_0=\{f_i\}_{i\in \IN{\n0  }}\in \cH^{\n0  }$ and 
$\ca= (\alpha_i)_{i\in \IN{k}} \in \R_{>0}^k\,$ 
such that $d-\rk \, S_{\cF_0} \le k$. Define $n=\n0+k$. Then
\ben
\item In what follows we say that
 $(\cF_0 \coma \ca)$ are initial data for the completion problem (CP). 
\item 
For these data we consider the sets
$$
\cC_\ca(\cF_0)=\big\{\, \{f_i\}_{i\in \IN{n}} 
\in \hil^n : \{f_i\}_{i\in \IN{\n0  }}= \cF_0 \py \|f_{\n0+i}\|^2
=\alpha_i \ \mbox{ for }  \ i \in \IN{k}\big\}\ ,
$$ 
$$
\py 
\cS\cC_\ca(\cF_0)=\{S_\cF:\ \cF\in \cC_\ca(\cF_0)\} \inc \matpos
\ . 
$$
\een
When the initial data $(\cF_0 \coma \ca)$ are fixed, 
we shall use throughout the paper the notations 
\beq
S_0 = S_{\cF_0} \ \ \coma \ \ \la = \la(S_0)\ \  \py\ \ n = \n0+k \ .
\EOEP\eeq
\end{fed}

\pausa
{\bf Problem:} (Optimal completions with prescribed norms 
with respect to  $P_f\,$) \ID \ and let $f:[0,\infty)\rightarrow \R$ be 
a strictly convex function. Construct all possible 
$\cF\in \cC_\ca(\cF_0)$ that are the minimizers of $P_f$ 
in $\cC_\ca(\cF_0)$. \EOE

\pausa
Our analysis of the completed frame $\cF= (\cF_0\coma \cG )$ will depend on $\cF$ 
through $S_\cF\,$. Hence, the following description of $\cS\cC_\ca(\cF_0)$ plays a central role in our approach.

\begin{pro}\label{con el la y mayo}
\ID .
Then 
$$
\cS\cC_\ca(\cF_0) =
\big\{S\in \matpos \, :\, 
S\geq S_{\cF_0}  \py \ca \prec \lambda(S-S_{\cF_0})  
\big\}
\ .
$$ 
\end{pro}
\proof 
Observe that if $\cF = (\cF_0\coma \cG ) \in \hil^n$ then 
$S_{\cF} = S_{\cF_0} + S_{\cG }\,$. 
Denote by $ S_0 = S_{\cF_0}$ and 
 $B = S-S_0\,$, for  $S\in \matpos$. 
Applying Proposition \ref{frame mayo} to the matrix $B$
(which must be nonnegative if $S\in \cS\cC_\ca(\cF_0)\,$), 
we get the equality of the sets. 
\QED

\begin{rem}[Optimal completion problem with prescribed norms: the feasible case]\label{caso feasible}
\ID\ . 
Denote by $S_0 = S_{\cF_0}\,$, $\la = \la (S_0)$ and $t=\tr\, \la + \tr \, \ca$. In \cite{MRS4} we introduced the following set 
$$ 
U_t(S_0\coma m)=\{S_0+B:\   B\in \matpos \, , \ \rk \, B 
\le d-m \ , \ \tr\,(S_0+B)\ =\  t\ \} \ ,
$$
where $m = d-k$. In \cite[Theorem 3.12]{MRS4} 
 it is shown that there exist $\prec$-minimizers in $U_t(S_{0}\coma m)$.  
Indeed, there exists $\nu=\nu(\la \coma m)\in (\R^d_{\geq 0})^\downarrow$ - that 
can be effectively computed by simple algorithms - such that $S\in U_t(S_{0}\coma m)$ is 
a $\prec$-minimizer if and only if $\la(S)=\nu$. 

\pausa
We say that the completion problem for $(\cF_0,\cb)$ is {\bf feasible} if $\mu\igdef\nu-\la$ satisfies that $\cb\prec\mu$, where $\nu=\nu(\la,m)$ is as above. In this case for any $S$ which is a $\prec$-minimizer in $U_t(S_0\coma m)$ it holds that $\la(S-S_0)=\mu^\downarrow$ and hence, by Proposition \ref{con el la y mayo}, we conclude that $S\in \cS\cC_\ca(\cF_0)$. Moreover, Proposition \ref{con el la y mayo} also shows that  $\cS\cC_\ca(\cF_0)\inc U_t(S_0\coma m)$ and therefore $S$ is a $\prec$-minimizer in $\cS\cC_\ca(\cF_0)$. In this case, as a consequence of the results in Section \ref{subsec 2.2}, 
any completion $\cF\in \cC_\ca(\cF_0)$ such that $S_\cF=S$ is a minimizer of $P_f$ for any convex function $f:[0,\infty)\rightarrow \R$. That is, in the feasible case we have structural solutions of the completion problem, in the sense that these solutions do not depend on the particular choice of convex potential considered.  

\pausa
Nevertheless, it is easy to construct examples in which the completion problem for $(\cF_0\coma
\ca)$ is not feasible.
For example, consider the frame $\cF_0\in \RS(7\coma 5)$ 
 whose synthesis operator is 
\beq\label{ejemF0}
T_{\cF_0}^* = \mbox{ \scriptsize $  
\left[
 \begin{array}{rrrrrrr}
    0.9202&   -0.7476&   -0.4674&    0.9164&    0.1621&    0.3172&   -0.5815\\
    0.4556&    0.0164&    0.0636&    1.0372&   -1.6172&    0.3688&    0.2559\\
   -0.0885&   -0.3495&   -0.9103&    0.3672&   -0.6706&   -0.9252&    0.6281\\
    0.1380&   -0.4672&   -0.6228&   -0.1660&    0.9419&    1.0760&    1.1687\\
    0.7082&    0.2412&   -0.1579&   -1.8922&   -0.4026&    0.1040&    1.6648\\
 \end{array}
 \right]$}
 \ .
 \eeq
In this case
$\la=\la(S_{\cF_0})=(9\coma 5\coma 4\coma 2\coma 1) $ and 
 $t_0 = \tr\,S_{\cF_0}= 21$.  Fix the data $n=9$ (hence $k=2$),  $\cb = (3.5 \coma 2)$ and notice that then
$t = t_0 + \tr\, \cb = 26.5$ and $m = d-k = 3$. 
Then, according to the results in \cite{MRS4} we know that the optimal 
spectrum for $U_t(S_0\coma m)$ is 
$\nu_{\la\coma m}(26.5) = (9\coma 5\coma 4.25 \coma 4.25 \coma 4)$. Therefore, we have that 
$\nu-\la=\mu=(2.25\coma 3.25)$ so that $\cb \not \prec\mu$, that is the completion problem $(\cF_0,\cb)$ is not feasible.

\pausa
The structure of the optimal completions with these norms was not known, even for the MSE.
In what follows we shall give a complete description of the optimal frame completions - with respect to an arbitrary convex potential - for this initial data (see Example \ref{ejemp Acha}). 
\EOE
\end{rem}

\section{The spectrum of the minimizers of $P_f$ on $\cC_\ca(\cF_0)$}\label{sec3}
\ID. 
Let $\mu\in \R_{\geq 0}^d$ be such that $\cb\prec \mu$. 
We consider the set 
$$
\cC_\ca(\cF_0 \coma \mu)
\igdef\{ \cF=(\cF_0\coma \cG )\in \cC_\ca(\cF_0):\ \lambda(S_1)
=\mu^\downarrow\} 
\inc\cC_\ca(\cF_0) \ .
$$
Notice that if $\cF=(\cF_0,\cG)$ then $S_\cF=S_{\cF_0}+S_\cG\,$. 
By Proposition \ref{con el la y mayo} we get the following partition:  
\beq\label{particion}
\cC_\ca(\cF_0)=\bigsqcup_{\mu\in \,\Gamma_d(\cb)} 
\cC_\ca(\cF_0 \coma \mu)
\peso{where}  
\quad \Gamma_d(\cb)\igdef  \{ \mu\in (\R_{\geq 0}^d)^\uparrow: \ \cb\prec \mu\}
\ .
\eeq

\begin{teo}\label{dale que va}
Consider the previous notations and 
fix $\mu = \mu\ua \in \Gamma_d(\cb)$. Then,
\begin{enumerate}
\item The set $\Lambda(\cC_\ca(\cF_0\coma \mu ))\igdef 
\{\lambda(S_\cF):\ \cF\in \cC_\ca(\cF_0\coma \mu )  \}$ is convex.
\item The vector  $\nu=(\lambda(S_{\cF_0})+\mu )^\downarrow$ is 
a $\prec$-minimizer in $\Lambda(\cC_\ca(\cF_0\coma \mu ))$.
\item If $\cF=(\cF_0\coma \cG )\in  \cC_\ca(\cF_0\coma \mu )$ 
is such that $\lambda(S_\cF)=\nu$ then $S_{\cF_0}$ and $S_{\cG }$ commute.
\end{enumerate}
\end{teo}
\begin{proof}
1. First notice that the set of all frame operators 
$S_{\cG }\in \matpos$ such that $\cF=(\cF_0\coma \cG )
\in \cC_\ca(\cF_0 \coma \mu )$ is closed under unitary equivalence.  
Indeed, if $U \in \matu$, then $U\, S_{\cG }\, U^*$ is the frame 
operator of the sequence $U\cdot \cG =\{Uf_i\}_{i=\n0+1}^n\,$.
Denote by $\la =  \lambda(S_{\cF_0})$. Therefore, 
it is straightforward to check that 
$$
\Lambda(\cC_\ca(\cF_0 \coma \mu ))=\{\lambda(C):\ C=A+B, \ A,\,B\in\matsa , 
\, \lambda(A)=\lambda \py \lambda(B)=\mu  \}\ .
$$ 
By Klyachko's theory on the sum of hermitian matrices with a given spectra \cite{Klyachko}
, we conclude that the set $\Lambda(\cC_\ca(\cF_0 \coma \mu ))$ is convex.

\pausa
2. Since the set of all frame operators $S_{\cG }\in \matpos$ such that 
$\cF=(\cF_0 \coma \cG )\in \cC_\ca(\cF_0 \coma \mu )$ is closed under 
unitary equivalence it is clear that $\nu\in \Lambda(\cC_\ca(\cF_0 \coma \mu ))$. 
On the other hand, given $\cF=(\cF_0\coma \cG )\in \cC_\ca(\cF_0 \coma \mu )$, 
then Lindskii's inequality (see Theorem \ref{LindsTeo}) states that the 
vector $\nu\prec\lambda(S_{\cF_0}+S_{\cG }) = \la (S_\cF)$. This establishes 
that $\nu$ is a $\prec$-minimizer in $\Lambda(\cC_\ca(\cF_0 \coma \mu ))$.

\pausa
3. This is a restatement of Theorem \ref{teoleq}.
\end{proof}

\begin{rem}\label{min OP}
Consider the previous notations and fix $\mu = \mu\ua\in \Gamma_d(\cb)$. 
Let $f:[0 \coma \infty)\rightarrow [0 \coma \infty)$ 
be a strictly convex function and let $P_f$ be the convex potential induced by $f$.  
By the results described in Section \ref{subsec 2.2} 
and Theorem \ref{dale que va} we see that, if   $\la=\la(S_{\cF_0})$ then 
\beq\label{eqpara cafop}\cF\in
{\rm argmin}\{P_f(\cG)  : \cG\in \cC_\ca(\cF_0 \coma \mu )\} 
\iff \lambda(S_\cF)=(\lambda+\mu )\da  =(\,\la\da +\mu \ua\,)\da\ .
\eeq 
That is, if we consider the partition of $\cC_\ca(\cF_0)$ described in Eq. \eqref{particion}, 
then in each slice $\cC_\ca(\cF_0 \coma \mu )$ the minimizers of the potential $P_f$ 
are characterized by the spectral condition  \eqref{eqpara cafop}.

\pausa
This shows that in order to search for global minimizers 
of $P_f$ on $\cC_\ca(\cF_0)$ we 
can restrict our attention to the set 
\beq \label{eqcafop} 
\cC_\ca^{\rm op}(\cF_0) \igdef 
\big\{\, \cF=(\cF_0\coma \cG )\in \cC_\ca(\cF_0):\ \lambda(S_{\cF})
= \big(\,\la(S_{\cF_0})+\la^\uparrow(S_{\cG })\, \big)^\downarrow \, \big\} \ .
\eeq
Indeed, Eqs. \eqref{particion} and  \eqref{eqpara cafop} 
show that if $\cF$ is a minimizer of $P_f$ in $\cC_\ca(\cF_0)$ 
then $\cF\in \cC_\ca^{\rm op}(\cF_0)$, i.e. 
\beq\label{en el op}
{\rm argmin} \, \{P_f(\cF): \cF\in \cC_\ca(\cF_0)\} \ = \ 
{\rm argmin} \, \{P_f(\cF): \cF\in \cC_\ca^{\rm op}(\cF_0)\} 
\ .
\eeq
Since the potential $P_f(\cF)$ depends on $\cF$ through the 
eigenvalues of $S_\cF$ we 
introduce the set 
\beq\label{defi scafop} 
\cS(\cC_\ca^{\rm op}(\cF_0)) \igdef
\{S_\cF:\ \cF\in \cC_\ca^{\rm op}(\cF_0)\}
\inc \matpos \ .
\eeq
Finally, for any $\la \in \R_{\ge0}^d \,$,  in what follows we shall also consider the set 
\beq
\Lambda_\ca^{\rm op}(\lambda) \igdef
\{\lambda\da +\mu :\mu \in \Gamma_d(\cb)\} = 
\{\lambda\da +\mu\ua : \ \mu \in \R^d_{\geq 0} \py \cb \prec \mu \}\ .
\EOEP
\eeq
\end{rem}

\begin{teo}\label{teocafop2}  
\ID.
Denote by  $\lambda=\lambda(S_{\cF_0})$. Then
\ben
\item The set $\Lambda_\ca^{\rm op}(\lambda)$ is compact and convex.
\item 
The spectral picture $\{\lambda(S_\cF):\ \cF\in \cC_\ca^{\rm op}(\cF_0) \}
=\{ \nu\,^\downarrow:\ \nu\in \Lambda_\ca^{\rm op}(\lambda)\}$.
\item If $\cF=(\cF_0 \coma \cG )\in\cC_\ca^{\rm op}(\cF_0)$, 
with $\lambda^\uparrow(S_{\cG }) =\mu$, then there exists 
$\{v_i: i\in \IN{d}\}$ an ONB
 of eigenvectors for $S_{\cF_0}\coma \la$  such that 
\beq\label{la BON2}
S_{\cG } = \sum_{i\in \IN{d}} \, \mu_i \cdot v_i \otimes v_i 
\py 
S_{\cF}=S_{\cF_0} + S_{\cG } = \sum_{i\in \IN{d}} \, (\la_i + \mu_i) \, v_i \otimes v_i \ .
\eeq
\een 
\end{teo}
\begin{proof} 1.  If $\nu\coma \ga\in \Lambda_\ca^{\rm op}(\lambda)$ 
then there exist $\mu \coma \rho\in \Gamma_d(\cb)$ 
such that $\nu=\lambda^\downarrow+\mu $, $\ga=\lambda^\downarrow+\rho$. 
Note that $\Gamma_d(\cb)$ is convex. Hence, 
if $t\in (0,1)$,  then 
$\mu_t \igdef t \,\mu + (1-t)\, \rho \in \Gamma_d(\cb)$ and  
$$
t \,\nu + (1- t)\, \ga= \lambda^\downarrow + 
t \,\mu + (1-t)\, \rho = 
\lambda^\downarrow + \mu_t \in \Lambda_\ca^{\rm op}(\lambda)
\ .
$$
Item 2. is an immediate consequence of the definitions of 
$\Lambda_\ca^{\rm op}(\lambda)$ and $\cC_\ca^{\rm op}(\cF_0)$. 

\pausa
3. Let $\cF=(\cF_0 \coma \cG )\in\cC_\ca(\cF_0)^{\rm op}$. Then the frame operator $S_{\cG }$ is an optimal matching matrix for $S_{\cF_0}$ in the sense of Eq. \ref{los S1} (see the Appendix).
Hence, the existence of an ONB  $\{v_i:\ i\in \IN{d}\}$ 
for $S_0\coma \la$ satisfying Eq. \eqref{la BON2} follows from
Theorem \ref{LA ONB}.
\end{proof}

\pausa
\ID. Recall that 
$\Gamma_d(\cb)= \{ \mu\in (\R_{\geq 0}^d)^\uparrow: \ \cb\prec \mu\}$. In what follows we use the following notation: if $f:[0,\infty)\rightarrow \R$ is a function we consider $F:\R_{\geq 0}^d \rightarrow \R$ given by $F(\gamma)=\sum_{i\in \IN{d}} f(\gamma_i)$, for $\gamma\in \R_{\geq 0}^d\,$.

\begin{teo}\label{teo sobre unico espectro}  
\ID \ and let $f:[0 \coma \infty)\rightarrow [0 \coma \infty)$
be a strictly convex function. Then there 
exists a vector $\mu(\lambda \coma \ca \coma f) =\mu = 
\mu ^\uparrow\in 
\Gamma_d(\cb)$ such that: 
\ben
\item $\cF=(\cF_0 \coma \cG )\in \cC_\ca(\cF_0)$ is a global minimizer of $P_f \iff
\cF\in \cC_\ca^{\rm op}(\cF_0)$ and $\lambda^\uparrow(S_{\cG })=\mu$.
\item If we let $\la = \la(S_{\cF_0})$ then $\mu$ is  uniquely determined by the conditions 
\beq 
\mu \in \Gamma_d(\cb) \py 
F(\la+\mu)=\min_{\gamma\in\Gamma_d(\cb)}F(\la+\gamma) =
\min_{\nu\in\Lambda_\ca^{\rm op}(\la)}F(\nu) \ .
\eeq
\item Moreover, $\mu$ also satisfies that 
\beq\label{pegaditos}
0< \mu_i=\mu_{i+1}  \implies \lambda_i= \lambda_{i+1}\ \peso{for every} i \in \IN{d-1}\ .
\eeq
\een
\end{teo}
\begin{proof}
Notice that the map $F:\R_{\geq 0}^d\rightarrow [0 \coma \infty)$
is also strictly convex, and it is invariant under permutations 
of the variables. Moreover, 
\beq\label{ecua house}
P_f(\cF)=\tr\, f(S_\cF)=F(\lambda(S_\cF))
\peso{for every}  \cF\in \cC_\ca(\cF_0) \ .
\eeq
Since $\Lambda_\ca^{\rm op}(\lambda)$ is compact and convex and $F$ is strictly convex 
then every local minimizer of $F$ on $\Lambda_\ca^{\rm op}(\lambda)$ coincide with 
a unique global minimizer denoted by $\nu=\nu(\ca \coma \lambda\coma f)
\in \Lambda_\ca^{\rm op}(\lambda)$. Define $\mu=\nu-\lambda$ and notice that, by construction of the set $\Lambda_\ca^{\rm op}(\lambda)$, $\mu=\mu\ua$ and $\cb\prec\mu$. 

\pausa
Recall that given $\cF=(\cF_0\coma \cG )\in \cC_\ca(\cF_0)$ then a necessary condition for $\cF$ to be a global minimizer of $P_f$ 
on $\cC_\ca(\cF_0)$ is that $\cF\in \cC_\ca^{\rm op}(\cF_0)$ (see Remark \ref{min OP}).
Hence, by item 2 in Theorem \ref{teocafop2}, the fact that $F$ 
is permutation invariant and Eq. \eqref{ecua house} we conclude 
that $\cF\in \cC_\ca(\cF_0)$ is a global minimizer of $P_f$ on 
$\cC_\ca(\cF_0)$ if and only if 
$$
\cF=(\cF_0\coma \cG )\in \cC_\ca^{\rm op}(\cF_0)
\py \lambda(S_\cF)=\big(\, \lambda+
\la^\uparrow(S_{\cG })\, \big)^\downarrow=\nu^\downarrow \ .
$$ 
Denote by  $\rho=\la^\uparrow(S_{\cG })$. 
Then  $\cb\prec\rho =\rho^\uparrow$ and hence 
$\lambda+\rho\in \Lambda_\ca^{\rm op}(\lambda)$ is a minimizer 
of $F \iff \lambda+\rho=\nu \iff \rho=\mu$.

\pausa
Assume now that $0< \mu_i=\mu_{i+1}$ but $\lambda_i>\lambda_{i+1}\, $ for some 
 $ i\in \IN{d-1}\,$. 
We denote by $\rho$ the vector obtained from $\mu$ be replacing the $i$-th and $(i+1)$-th 
entries of $\mu$ by 
\bce 
$\rho_i=\mu_i-\varepsilon$  \ and \ $\rho_{i+1}=\mu_{i+1}+\varepsilon$ , 
 \ where \ $0<\varepsilon < \min\{ \frac{\lambda_i-\lambda_{i+1}}{2}\, , \, \mu_i\}$ .
\ece 
Although it is possible that $\rho \neq \rho ^\uparrow$, the facts that  
$(\mu_i\coma \mu_{i+1})\prec(\rho_i\coma \rho_{i+1})$ and $\mu_j=\rho_j$ 
for every $j\in(\IN{d}\setminus\{i\coma i+1\})$ imply, by Remark \ref{rem doublystohasctic}, 
that $\mu\prec\rho$ and hence $\cb\prec\mu\prec\rho$. Using 
Proposition \ref{con el la y mayo} and fixing an ONB for $S_{\cF_0} \coma \la$,   
we deduce that there 
exists $\cF\,'=(\cF_0 \coma \cG\,')\in \cC_\ca(\cF_0)$ such that $\lambda(S_{\cG\,'})
=\rho^\downarrow$ and $\lambda(S_{\cF\,'})=(\lambda+\rho)^\downarrow$.
Recall that $\nu = \la + \mu$. Note that
$$
\nu_i=\lambda_i+\mu_i> \lambda_i+\rho_i >\lambda_{i+1}+\rho_{i+1} 
>\lambda_{i+1}+\mu_{i+1}=\nu_{i+1}  \ , 
$$ 
while  $\nu_j=\la_j+\mu_j=\la_j+\rho_j$ 
for every $j\in(\IN{d}\setminus\{i\coma i+1\})$.  Then, by 
Remark \ref{rem doublystohasctic}, we conclude that 
$\lambda+\rho\prec\nu$ and $(\lambda+\rho)^\downarrow\neq\nu^\downarrow$. 
Hence, if $f$ is strictly convex the previous facts imply that 
$P_f(\cF\,')<F(\nu)$, which contradicts the minimality of $\nu$ of the first part of this proof.
\end{proof}

\begin{cor}\label{loqueimporta} 
Let $\cF=(\cF_0 \coma \cG )\in \cC_\ca(\cF_0)$ be a global minimizer of $P_f\,$. Then 
\ben
\item If \ $z\in\sigma(S_\cG)\setminus\trivial$ then there exists 
$w\in \sigma(S_0)$ such that $\ker(S_\cG-z)\inc  \ker(S_0-w)$. 
\item In particular, if $P$ denotes a sub-projection of 
the spectral projection $P(z)$ of $S_\cG$ 
onto its eigenspace $\ker(S_\cG-z)$, then $P$ and $S_0$ commute.
\een
\end{cor}
\begin{proof}
By Remark \ref{min OP}, the $P_f\,$-minimality of  $\cF=(\cF_0 \coma \cG )$ in 
$ \cC_\ca(\cF_0)$ implies that 
$\cF\in \cC_\ca^{\rm op}(\cF_0)$. Then,  by Theorem \ref{teocafop2}, 
there exists $\{v_i: i\in \IN{d}\}$ an ONB
 of eigenvectors for $S_0\coma \la=\la(S_0)$  such that Eq. \eqref{la BON2} holds.
Denote by $S_1 = S_\cG\,$, $\mu=\la^\uparrow(S_1)$  and fix $z\in\sigma(S_1)\setminus\trivial$. 
Consider the indexes  
$$
m(z) = \min \{i \in \IN{d}:\mu_i=z\} \py 
M(z) = \max \{i \in \IN{d}:\mu_i=z\}  \ .
$$
By Eq. \eqref{pegaditos} in Theorem \ref{teo sobre unico espectro} 
we know that there exists $w\in\sigma(S_0)$ such 
that $\lambda_i=w$  
for every $m(z)\leq i\leq M(z)$. 
Then, we can use  Eq. \eqref{la BON2} and deduce that 
$$
\ker(S_\cG-z)= \gen \{v_i: \ m(z)\leq i\leq M(z)\} 
\inc  \ker(S_0-w)  \ .
$$ 
Therefore, any projection $P$ as in  item 2 must satisfy that 
$P\cdot S_0=S_0\cdot P=w P$.
\end{proof}

\begin{rem}[First reduction of the optimal CP problem] \ID \ and let 
$f:[0,\infty)\rightarrow \R$ be a strictly convex function. Consider the 
compact convex set $\Lambda_\ca^{\rm op}(\la)\inc \R_{\geq 0}^d$ and 
define the strictly convex function $F:\Lambda_\ca^{\rm op}(\la)\rightarrow \R$. 
Therefore $F$ is continuous and hence 
\beq\label{convex prog} 
\exists \, ! \ \ \text{argmin}  \ \{ F(x) : {x\in \Lambda_\ca^{\rm op}(\la)}\}
=\nu
\eeq
Theorem \ref{teo sobre unico espectro} states that
$\mu(\la\coma \ca\coma f)=\nu-\la$. 
Thus, $\cF=(\cF_0 \coma \cG)\in \cC_\ca^{\rm op}(\cF_0)$ is an optimal completion 
with respect to $P_f$ if and only if $\la(S_\cG)=\nu-\la$. Thus, the minimization 
problem in Eq. \eqref{convex prog} constitutes a reduction of the optimization 
problem, that in turn can be 
attacked with several numerical tools in concrete examples. Notice that 
$\Lambda_\ca^{\rm op}(\la)$ has a natural and explicit description in terms 
of majorization, which is an algorithmic notion.

\pausa
In the next sections we develop a different approach to the computation 
of minimizers of $P_f$ in $\cC_\ca(\cF_0)$ (see Section \ref{sec consec conjetur}).
\EOE
\end{rem}

\section{Local minimizers of $P_f(\cdot)$ on $\cC_\ca^{\rm op}(\cF_0)$}\label{Sec4}

In applied situations it is quite useful to understand the structure of local minimizers of objective functions. In our case, the study of local minimizers allows us to give a detailed description of the geometrical structure of global minimizers. 
We shall consider two different topologies on the set 
$\cC_\ca(\cF_0)$. On the one hand, we consider the 
pseudo-metric $d_S$ given by 
$$
d_S(\cF \coma \cF\, ')=\|S_\cF-S_{\cF\, '}\|  \ ,
$$ 
where $\|\cdot\|$ denotes the spectral norm on $\mat$. 
On the other hand, we also consider the punctual metric $d_P$ given by 
$$
d_P(\cF\coma \cF\, ')=\|T_\cF-T_{\cF\, '}\| \ , 
$$ 
where as before $\|\cdot \|$ denotes the spectral norm. It is clear that the 
topology induced by $d_P$ is strictly stronger in the sense that: 
if $\cF_n\xrightarrow[n]{d_P}\cF$ then $\cF_n\xrightarrow[n]{d_S}\cF$, 
while the converse is false. Hence, $d_S$-local minimizers are 
also $d_P$-local minimizers. 

\pausa
Let $f:[0 \coma \infty) \rightarrow [0 \coma \infty)$ be a strictly 
convex function. Recall from Remark \ref{min OP} that  global minimizers 
of $P_f$ on $\cC_\ca(\cF_0)$ actually lie in $\cC_\ca^{\rm op}(\cF_0)$.  
Therefore we shall focus our interest in the geometrical and spectral structure of local minimizers $\cF\in \cC_\ca^{\rm op}(\cF_0)$ of $P_f\,$.

\subsection{The $d_S$-local minimizers on $\cC_\ca^{\rm op}(\cF_0)$ are global minimizers}

\begin{rem}\label{curvas de autovectores}
Let $\cB_1= \{u_j\}_{j\in \IN{d}}$ and $\cB_2= 
\{v_j\}_{j\in \IN{d}}$ be two ONB for $\C^d$. 
Then there exist  continuous curves 
$w_j:[0,1]\rightarrow \C^d$ ($j\in \IN{d}$)
such that $w_j(0)=v_j \coma w_j(1)=u_j$ and such that 
$\{w_j(t)\}_{i\in \IN{d}}$ is an ONB for $\C^d$ 
for every  $t\in [0,1]$. 

\pausa
In fact, given the unitary matrix $U\in\matud$ such that 
$U\, v_j = u_j$ for every $j\in \IN{d}\,$, 
there exists a unique $X\in \gld$ with  $\|X\|\leq 2\,\pi$ such that $e^{\,i\,X}=U$. 
Hence the continuous curve $\gamma_U:[0,1]\rightarrow \matud$ given by 
$\gamma_U(t)=e^{\,i\,t\,X}$ joins $\gamma_U(0)=I$ with $\gamma_U(1)=U$. Thus, the continuous 
curves $w_j(t)=\gamma_U(t)\,v_j$ enjoy the mentioned  properties. 

\pausa
Assume further that there exists $S\in \matpos$ such 
that both $\cB_1 $ and $\cB_2 $ 
are ONB of eigenvectors for $S \coma \la = \la(S)$. 
Hence, by Eq. \eqref{BON S la} we have that 
$$
S=\sum_{i\in \IN{d}}\lambda_i \ u_i\otimes u_i
=\sum_{i\in \IN{d}}\lambda_i \ v_i\otimes v_i\ .
$$ 
In this case, it is easy to see that 
the unitary $U \in \matud$ such that $U\, v_j = u_j$ for every $j\in \IN{d}\,$ should also satisfy 
that $S\, U = U\,S$ and that $\ga_U(t)\, S = S \, \ga_U(t)$ for every $t \in [0,1]$.

\pausa
Then the continuous curves $w_i:[0,1]\rightarrow \C^d$ previously constructed 
also satisfy that the basis 
$\{w_i(t)\}_{i\in \IN{d}}$ is an ONB of eigenvectors for $S$, $\lambda$, for every  $t\in [0,1]$.
In other words, for every $t\in [0,1]$ we have the identity 
\beq
S= \ga_U(t)\, S \, \ga_U(t)^* = 
\sum_{i\in \IN{d}}\lambda_i \ \ga_U(t)\, v_i\otimes \ga_U(t)\, v_i = 
\sum_{i\in \IN{d}}\lambda_i \ w_i(t)\otimes w_i(t) \ . 
 \EOEP
 \eeq
\end{rem}

\begin{teo}\label{ds min loc}
\ID \ and fix  a strictly convex function
 $f:[0\coma \infty)\to [0\coma \infty)$. 
Then every $d_S$-local minimizer of $P_f$ on $\cC_\ca^{\rm op}(\cF_0)$
must be a {\bf global} minimizer. 
\end{teo}

\begin{proof}
Let $\cF\,' =(\cF_0 \coma \cG\,' )$ be a global minimizer of 
$P_f$ in $\cC_\ca^{\rm op}(\cF_0)$, so that $\la^\uparrow(S_{\cG\,'})=\mu
= \mu(\lambda \coma \ca\coma f)$  
the vector of Theorem \ref{teo sobre unico espectro}. 
On the other hand take  $\cF=(\cF_0 \coma \cG )\in\cC_\ca^{\rm op}(\cF_0)$  
a $d_S$-local minimizer of $P_f$ on $\cC_\ca^{\rm op}(\cF_0)$. 
We denote by $ S_{\cF_0}= S_0\,$,  $\lambda(S_0)=\lambda$ and $\la^\uparrow(S_\cG)=\rho$.
Then $\mu \coma \rho\in  \Gamma_d(\cb) $  and by Theorem \ref{teocafop2} 
(applied to both $\cF$ and $\cF\,'$) there exist 
two ONB's $\{ u_i:\ i\in \IN{d}\}$ and $\{ v_i:\ i\in \IN{d}\}$
such that 
\beq \label{ecuac1}
\barr{rl}
S_0& =\suml_{i\in \IN{d}}\lambda_i\ u_i\otimes u_i= \suml_{i\in \IN{d}}\lambda_i\ v_i
\otimes v_i\quad , \quad   S_{\cF\,'}=\suml_{i\in \IN{d}}(\lambda_i+\mu_i)\ u_i\otimes u_i\ ,  \\&\\
& \py   S_{\cF}=\suml_{i\in \IN{d}}(\lambda_i+\rho_i)\ v_i\otimes v_i\ .
\earr
\eeq  
Therefore, by Remark \ref{curvas de autovectores}, 
there exists a family of continuous curves 
$w_i:[0,1]\rightarrow \hil$  such that 
$w_i(0)=v_i\,$ and  $w_i(1)=u_i$ for every $i\in \IN{d}$ and such that 
$\{w_i(t):\ i\in\IN{d}\}$ is an ONB for $S_0$ and $\la$ 
for every $t\in [0,1]$.
Define the continuous curve $\s:[0,1]\rightarrow\matpos$ given by
$$ 
\barr{rl}
\s(t)& =\suml_{i\in\IN{d}} (\lambda_i+t\cdot \mu_i+(1-t)\cdot \rho_i) \ w_i(t)\otimes w_i(t)\\&\\
& =S_0+ \suml_{i\in\IN{d}} (t\cdot \mu_i+(1-t)\cdot \rho_i) \ w_i(t)\otimes w_i(t) \peso{for every}  t\in [0,1]\ . \earr
$$ 
It is clear that $\s(0)=S_{\cF}$ and 
$\s(1)=S_{\cF\,'}$. 
We claim that $\s(t)\in \cS(\cC_\ca^{\rm op}(\cF_0))$ 
for every $t\in[0,1]$. Indeed, notice that 
$(t\cdot \mu+(1-t)\cdot \rho)^\uparrow=t\cdot \mu+(1-t)\cdot \rho$ 
and therefore $\cb\prec(t\cdot \mu+(1-t)\cdot \rho)$ 
for $t\in [0,1]$. Hence, there is a map $\s_1:[0,1]\to 
\matpos$ such that for every $t\in [0,1]$
$$
\s(t)=S_0+\s_1(t) \ \ , \ \ \ \cb\prec\lambda(\s_1(t)) 
\py \lambda(\s(t))=\big(\lambda(S_0)+\la\ua(\s_1(t)\,)\,\big)\da\ .
$$
These last facts prove our claim. Notice that then $h(t)=P_f(\s(t))\, , \ t\in[0,1]$ is a strictly convex function that has local minima 
at $t=0$ and $t=1$, i.e. $h$ is constant. 
Then $P_f(\cF)=h(0)= h(1)= P_f(\cF\,')$, and $\cF$ 
is another global minimizer. 
\end{proof}

\begin{rem}[On $d_S$-local minimizers in $\cC_\ca(\cF_0 \coma \mu )$ and 
$\cC_\ca(\cF_0)$ and a lifting problem]  
The previous result raises the question about the spectral structure of 
$d_S$-local minima of $P_f$ on $\cC_\ca(\cF_0 \coma \mu )$ or on $\cC_\ca(\cF_0)$. 
Indeed, let $\cb\prec \mu =\mu ^\uparrow$ and consider $\cF$ a $d_S$-local minimizer 
in $\cC_\ca(\cF_0 \coma \mu )$. As shown in Theorem \ref{dale que va} the set 
$\Lambda(\cC_\ca(\cF_0 \coma \mu ))$ is convex. Therefore 
$\lambda(t)=t\cdot \lambda(S_\cF)+(1-t)\cdot (\lambda_0+\mu )^\downarrow\in 
\Lambda(\cC_\ca(\cF_0 \coma \mu ))$ for $t\in [0,1]$ is a continuous curve. 

\pausa
Assume that we can lift the curve $\lambda(\cdot)$ to a curve 
in $\cS( \cC_\ca(\cF_0 \coma \mu ))$ i.e., assume that there exists a 
continuous curve 
\beq \label{eclift} 
\s :[0,1]\rightarrow \cS( \cC_\ca(\cF_0 \coma \mu )) 
\peso{ such that } \lambda(\s(t))=\lambda(t)\ .
\eeq  	 
Then, we could argue as in Theorem \ref{ds min loc} above and 
conclude that $\lambda(S_\cF)= (\lambda_0+\mu )^\downarrow$, 
which in turn would also imply that $d_S$-local minimizers 
of $P_f$ on $\cC_\ca(\cF_0)$ are also global minimizers. 
Although we conjecture that the lifting problem of 
Eq. \eqref{eclift} has a solution, we are not able 
to show that such a solution exists at this time.
\EOE
\end{rem}

\subsection{A geometrical approach for $d_P$-local minimizers on $\cC_\ca^{\rm op}(\cF_0)$}

In what follows we consider a geometrical approach to the study of $d_P$-local minimizers. Our results are based on
a perturbation result for finite sequences of vectors, which follows from the work in \cite{MRS}. In order to describe the general setting, we begin by considering some well known facts from differential geometry. In what follows we consider the unitary group of a complex and finite dimensional inner product space $\cR$, denoted $\cU(\cR)$, together with its natural differential geometric (Lie) structure. 

\pausa
\ID. Fix $\cF=(\cF_0 \coma \cG )
= \{f_i\}_{i=1}^{n}\in \cC_\ca(\cF_0)$, where $n=k+\n0\,$, 
$\cR =R(S_{\cG }) = \gen \{\cG \} \inc\C^d$, and    
$\tau=\tr \,\cb=\sum_{i=1}^{k}\alpha_i>0$. Consider the real vector space 
\beq\label{poRt}
\hil_d(\cR)^\tau=\{S\in\matsad:\ R(S)\inc\cR \, , \ \tr \,S=\tau\} \ , 
\eeq
the cone $\poRt = \hil_d(\cR)^\tau \cap \matpos$, 
and the affine manifold 
$$
S_{\cF_0}+\hil_d(\cR)^\tau=\{S_{\cF_0}+S:\ S\in\hil_d(\cR)^\tau \}\inc\matsad\ .
$$
We define the smooth (and so $d_P$-continuous) map 
\beq\label{La Phi}
\Phi_\cF:\cU(\cR)^k\rightarrow \cC_\ca(\cF_0) 
\inc \cH^n \peso{given by} \Phi_\cF(U_i)_{i=1}^k
=\{f_i\}_{i=1}^{\n0}\cup \{ U_i f_{i+\n0}\}_{i=1}^k\ .
\eeq
Finally, we consider the smooth map $\Psi_\cF:\cU(\cR)^k\rightarrow S_{\cF_0}+\hil_d(\cR)^\tau$ given by 
\beq\label{defi Psi}
\Psi_\cF(U_i)_{i=1}^k =S_{\cF_0}+\sum_{i=\n0+1}^n U_if_i\otimes U_if_i=S_{\cF\, '} 
\peso {where} \cF\, '=\Phi_\cF(U_i)_{i=1}^k \ .
\eeq 
Let us denote by $I^k=(I,\ldots,I)\in \cU(\cR)^k$. 
It turns out that in several cases (indeed, in a generic case) the map $\Psi_\cF$ is an open map (in $S_{\cF_0}+\hil_d(\cR)^\tau$) around  
$\Psi_\cF(I^k)=S_\cF\,$. In order to characterize this situation we introduce the following notion. 

\pausa

\begin{fed}\label{irre} \rm
Given a sequence $\cG=\{g_i\}_{i\in \IN{n}}$ in $\C^d$ we say that 
$\cG$ is  {\bf irreducible} if it can not be partitioned into two mutually 
orthogonal subsequences. 
\EOE
\end{fed}

\begin{rem}
\ID. 
Fix $\cF=(\cF_0 \coma \cG )\in \cC_\ca(\cF_0)$. Denote by  $n=k+\n0$ and 
$\cR =R(S_{\cG }) = \gen \{\cG \} \inc\C^d$.  
Consider the map $\Psi_\cF:\cU(\cR)^k\rightarrow S_{\cF_0}+\hil_d(\cR)^\tau$ defined 
in Eq. \eqref{defi Psi}.

\pausa
In \cite{MRS} we have characterized when the map $\Psi_\cF$ is a submersion in terms of certain commutant. Indeed, let $L_d(\cR)$ denote the (non unital) $*$-subalgebra of $\mat$ that contains all $T\in\mat$ such that $T=P_\cR\, T\, P_\cR\,$, where $P_\cR$ denotes the orthogonal projection onto $\cR$. Then, an immediate application of 
\cite[Theorem 4.2.1.]{MRS} shows that $\Psi_\cF$ is a submersion if and only if the 
local commutant  
\beq\label{def mf1}
\mathcal M(\cG )\igdef  \{f_i\otimes f_i:\ \n0+1\leq i\leq n\}\, ' \cap  L_d(\cR) 
\eeq
is trivial, i.e. $\mathcal M(\cG )=\C\cdot P_\cR \,$. 
Equivalently, $\Psi_\cF$ is a submersion iff any 
$A\in  L_d(\cR)$ such that $A\,f_i=a_i\,f_i$ for some $a_i\in \C$, $\n0+1\leq i\leq n$, 
must be  $A=a\,P_\cR$ for some $a\in \C$. It is straightforward to show that this last condition on the family $\{f_i\otimes f_i:\ \n0+1\leq i\leq n\}$ is equivalent to the fact that the sequence $\cG $ is irreducible, in the sense defined above.
Thus, we have proved the following statement:
\EOE
\end{rem}

\begin{pro}\label{teo inmersion} \rm
\ID. 
Fix $\cF=(\cF_0 \coma \cG )\in \cC_\ca(\cF_0)$. Denote by  $n=k+\n0$ and 
$\cR =R(S_{\cG }) = \gen \{\cG \} \inc\C^d$.  Then the following statements are equivalent:
\begin{enumerate}
\item The map $\Psi_\cF$ of Eq. \eqref{defi Psi} is a submersion at $I^k\in\cU(\cR)^k$.
\item The sequence $\cG $ is irreducible.
\end{enumerate} In this case, the image of $\Psi_\cF$ contains an open neighborhood of $\Psi_\cF(I^k)=S_\cF$ in $S_{\cF_0}+\hil_d(\cR)^\tau$. Hence, $\Psi_\cF$ admits a smooth 
local cross section $\psi$ around $S_\cF$ such that $\psi(S_\cF) =I^k\,$. 
\qed
\end{pro}

\pausa 
Next we state a convenient reformulation of Proposition \ref{teo inmersion}, 
in terms of the distance $d_P\,$.

\begin{cor}\label{coro sec locales dp}
\ID. Consider the smooth map 
\beq\label{La S}
\cS:\cC_\ca(\cF_0)\rightarrow S_{\cF_0}+\hil_d(\cR)^\tau 
\peso{given by} \cS(\cF)=S_{\cF}
=S_{\cF_0}+ S_{\cG} 
 \eeq
for every $\cF=(\cF_0 \coma \cG )\in \cC_\ca(\cF_0)$. 
If we assume that a point $\cF=(\cF_0 \coma \cG )\in \cC_\ca(\cF_0)$
satisfies that the sequence $\cG $ is {\bf irreducible}, then 
\begin{enumerate}
\item The image of $\cS$ contains an open neighborhood of $S_\cF\,$
in $S_{\cF_0}+\hil_d(\cR)^\tau \,$.
\item The map $\cS$ has a $d_P$-continuous local cross section $\varphi$ 
around $ S_\cF$ such that $\varphi(S_\cF)=  \cF\,$.
\end{enumerate} 
\end{cor}
\begin{proof}
Just define the $d_P$-continuous  local cross section  
$\varphi = \Phi_\cF \circ \psi$, 
where $\psi$ is the smooth local cross section for $\Psi_\cF$  of 
Proposition  \ref{teo inmersion} and $\Phi_\cF$ is the map  of Eq. \eqref{La Phi}. 
\end{proof}

\subsection{The $d_P$-local minimizers of $P_f$ in 
$\cC_\ca^{\rm op}(\cF_0)$ are frames for $\hil$}\label{sec4.3}

\begin{fed} \rm
Let $\cF=\{f_i\}_{i\in \IN{k}}\inc \cH^k$. 
A {\bf partition of $\cF$ into irreducible subsequences} is a family 
$\{\cF_i\}_{i\in\IN{p}}$ given by a partition 
$\Pi=\{J_i\}_{i\in \IN{p}}$ of the index set $\IN{k}$ in such 
a way that each $\cF_i=\{f_j\}_{j\in J_i}\,$ satisfies that: 
\bit
\item The subspaces $W_i =\gen \{\cF_i\}$ ($i\in \IN{p}$) are mutually orthogonal. 
\item Each subfamily $\cF_i$ ($i\in \IN{p}$) is irreducible. \EOE
\eit
\end{fed}

\pausa
Notice that any sequence $\cF=\{f_i\}_{i\in \IN{k}}\inc \cH^k$ has a unique such partition. 
To see this, consider the subspace 
$\cR = \gen \{\cF\} \inc\C^d$ and the (non-unital) $*$-subalgebra 
$\cM(\cF) =\{f_i\otimes f_i:\ i \in \IN{k}\}\, ' \cap  L_d(\cR) $.  
If $\cF$ is not irreducible,  then $\cM(\cF)$ contains a unique sequence of minimal 
orthogonal projections $\{Q_i\}_{i\in\IN{p}}$ such that $Q_i\,Q_j=0$ 
for $i,\,j\in \IN{p}$ such that $i\neq j$ and $\sum_{i\in\IN{p}}Q_i=P_\cR\,$. 
In this case,  we have that 
$$
Q_i\,f_j=\varepsilon(i,j) \,f_j  \peso{for every  \ $ i\in\IN{p}$\  and \ $j\in \IN{k}$ ,}
$$ 
where $\varepsilon(i,j)\in\{0,1\}$. 
Let $J_i=\{j\in \IN{k}:\ \varepsilon(i,j)=1 \}$ for $i\in \IN{p}\,$. Let $\Pi=\{J_i\}_{i\in\IN{p}}$.  The fact that 
$\sum_{i\in\IN{p}}Q_i=P_\cR$ implies that $\Pi$ is a partition of $\IN{p}\,$. 
The fact that  $\{Q_i\}_{i\in\IN{p}}$ is a family of mutually orthogonal 
projections imply that the subspaces $W_i= \gen \{f_j:\ j\in J_i\} =R(Q_i)$ 
are mutually orthogonal, while the fact that each $Q_i$ 
is a minimal projection in $\cM(\cF)$ implies that each 
$\cF_i=\{f_j\}_{j\in J_i}$ is irreducible. Then 
$\Pi=\{J_i\}_{i\in\IN{p}}$ has the desired properties.

\begin{lem}\label{sin F0}
Let $f:[0 \coma \infty)\rightarrow [0 \coma \infty)$ be a strictly convex function and let
$\{a_i\}_{i\in \IN{n}}\in \R_{>0}^n$ for some $n\geq d$. If $\cF=\{f_i\}_{i\in \IN{n}}$ is a $d_P$-local minimizer of $P_f$ in the set 
$$
\mathcal B(\ca)=\{\cG=\{g_i\}_{i\in \IN{n}}\in \hil^n: \ \|g_i\|^2=a_i\, , \ i\in \IN{n}\} \ ,
$$ 
then $\cF$ is a frame for $\hil$.
\end{lem}
\begin{proof}
Let $\Pi= \{J_i\}_{i\in \IN{p}}$ be a partition of $\IN{n}$ such that, if 
$\cF_i=\{f_j\}_{j\in J_i}$ for  $i\in \IN{p}\,$, then $\{\cF_i\}_{i\in \IN{p}}$ is a partition of $\cF$ into irreducible subsequences.
Recall that in this case the subspaces $W_i \igdef \gen \{\cF_i\}$ ($i\in \IN{p}$) are mutually orthogonal. Hence, it is easy to see that each subfamily $\cF_i$ is a $d_P$-local 
minimizer of $P_f$ in the set 
$$ 
\{\{g_j\}_{j\in J_i}:\ g_j\in W_i\, , \ \|g_i\|=\|f_i\|\, , \ j\in J_i \}\, .
$$ 
By \cite[Corollary 3]{MR}  and the properties of $\Pi$, each 
$\cF_i$ is a $c_i$-tight frame 
for $W_i\,$, for some $c_i>0$, $i\in \IN{p}\,$. 
Therefore 
$$
S_\cF =\suml_{i\in\IN{p}} S_{\cF_i} = \suml_{i\in\IN{p}}c_i\,P_{W_i}   \ . 
$$ 
Notice that, in particular, $S_\cF\, f_j=c_i\,f_j$ for every $j\in J_i\,$.

\pausa
Assume now that $\cF$ is not a frame for $\hil$. Then, there 
exists $i\in\IN{p}$ and $q,\, s\in J_i$ such that 
$\langle f_q,f_s\rangle \neq 0$, because otherwise $\cF$ would be 
a sequence of mutually orthogonal vectors and then, since $n\geq d$ 
then we would have $n=d$ and that span $\cF=\hil$. In particular, 
for this choice of indexes 
we have that $a_s=\|f_s\|^2<c_i\,$, because
$$
c_i\,\|f_s\|^2=\langle S_\cF\,f_s\coma f_s\rangle \geq 
|\langle f_s\coma f_s\rangle|^2
+ |\langle f_s\coma  f_q\rangle|^2=(\|f_s\|^2
+\frac{|\langle f_s\coma f_s\rangle|^2}{\|f_s\|^2})\,\|f_s\|^2 \ .
$$
We are assuming that $\ker S_\cF\, \neq \trivial$. Hence 
there exists $g\in \ker S_\cF\,$ with $\|g\|=\|f_s\|$. 
Let 
$$
f_s(t)=\cos(t)\cdot f_s+\sin(t)\cdot g  \peso{for every}  t\in [0,1] \ ,
$$ 
so that $f_s(0)=f_s$ and $f_s(1)=g$. 
Notice that 
$\|f_s(t)\|=\|f_s\|$ for every $t\in [0,1]$. 
Let $\cF(t)$ be the sequence obtained from $\cF$ by replacing $f_s$ by $f_s(t)$ and 
let $\s(t)$ denote the frame operator of $\cF(t)$, for 
each $t\in [0,1]$. Then 
$$ 
\s(t)=[S_{\cF}- (f_s\otimes f_s)] + f_s(t)\otimes f_s(t)
\peso{for every} t\in [0,1]\ .
$$ 
The inequality $a_s=\|f_s\|^2<c_i\,$ implies 
that $S_{\cF}- (f_s\otimes f_s)\in \matpos$ and also that 
$R(S_{\cF}- (f_s\otimes f_s))=R(S_\cF)$. 
Indeed, $S_{\cF}- (f_s\otimes f_s)=[a_s^{-1}\,(c_i-a_s)]\cdot f_s\otimes f_s+S\,'$ with $S\,'\in \matpos$; in this case  $\lambda(S\,')$ is 
obtained from $\lambda(S_\cF\,)$ by setting one of the 
occurrences of $c_i$ in $\lambda(S)$ equal to $0$, and 
$f_s\in\ker S\,'$. 
Thus, 
\beq \label{ecuaorto}
\s(t)=S\,'+ [a_s^{-1}\,(c_i-a_s)\cdot f_s \otimes f_s
+ f_s(t)\otimes f_s(t)] \peso{with} 
f_s \coma  f_s(t)\in \ker S\,' \ , 
\eeq 
for every $t\in [0,1]$. 
Using again the inequality $a_s=\|f_s\|^2<c_i\,$, let us define 
$$
\lambda(t)=\lambda( [a_s^{-1}\,(c_i-a_s)]\cdot f_s \otimes 
f_s+ f_s(t)\otimes f_s(t))=
(\lambda_1(t)\coma \lambda_2(t)\coma 0\coma \ldots\coma 0)
\in (\R_{\geq 0} ^d)^\downarrow \ .
$$ 
Then $\lambda(0)=(c_i\coma 0\coma \ldots\coma 0)$, 
$\lambda(1)=(c_i-a_s\coma  a_s\coma 0\coma \ldots\coma 0)^\downarrow$ 
and $\lambda_2(t)>0$ for $t>0$. 
Then there exists $t_0\in (0,1)$ such that for $0<t<t_0$, $\lambda_2(t)<\eps$ for $\eps >0$ such that 
$\eps<\min_{1\leq j\leq p}c_j$ and 
$\eps<\lambda_1(t)=(c_i-\lambda_2(t))$. By the previous 
remarks, it follows that $\la(\s(t))$ is obtained from 
$\lambda(S_\cF\,)$ by replacing one occurrence of $c_i$ by 
$\lambda_1(t)$ and one occurrence of $0$ by $\lambda_2(t)$. 
Therefore, if $r=\rk \,S_\cF\,$ then $\la_j(\s\,(t))
\leq \lambda_j(S_\cF\,)$ for $1\leq j\leq r$ and 
$\tr \, S_\cF\,=\sum_{j=1}^{r+1}\lambda_j(\s\,(t))=\tr \, 
\s\,(t)$ imply 
that $\lambda(\s\,(t))\prec \lambda(S_\cF\,)$ for $0<t<t_0\,$.

\pausa
These facts show that $\cF(t)$ converges with respect to the $d_P$-metric as $t\rightarrow 0^+$, while $P_f(\cF(t))<P_f(\cF)$ for $t\in (0,t_0)$. This contradicts the assumption that $\cF$ is a $d_P$-local minimum of $P_f$ and thus we should have that $R(S_\cF)=\hil$, i.e. $\cF$ is a frame.
\end{proof}

\begin{teo} \ID. 
Let $\cF=(\cF_0 \coma \cG )\in \cC_\ca^{\rm op}(\cF_0)$ be a $d_P$-local minimizer of 
$P_f$ on $\cC_\ca^{\rm op}(\cF_0)$, for some 
strictly convex function $f:[0 \coma \infty)\rightarrow [0 \coma \infty)$. 
Then $\cF$ is a frame, i.e. $S= S_\cF \in \gld$.
\end{teo}
\proof
Denote by $S_0 = S_{\cF_0}\,$, $\lambda(S_0)=\la =  \la^\downarrow$, 
$S_1 = S_{\cG }$ and $\lambda(S_1)=\mu^\downarrow$ for some 
$\cb\prec\mu=\mu^\uparrow$. 
Since $\cF=(\cF_0 \coma \cG )\in\cC_\ca^{\rm op}(\cF_0)$, 
by Theorem \ref{teocafop2} there exists an ONB
$\{v_i: i\in \IN{d}\}$ of eigenvectors for $S_0 \coma \la$ such that 
$S= S_0 + S_1 = \sum_{i\in \IN{d}} \, (\la_i + \mu_i) \, v_i \otimes v_i \,$.  
If $S \notin \gld$, let
\beq\label{el r de mu}
r = \max \{i\in \IN{d}: \la_i \neq 0\} < 
\min  \{j\in \IN{d}: \mu_j \neq 0\}  -1 \ . 
\eeq
Then $\cH_r = \gen\{v_i: i >r\}= \ker S_0\,$, and
 $S_1$ acts on $\cH_r\,$. 
The minimality of $\cF$ in $\cC_\ca^{\rm op}(\cF_0) $ 
implies that $\cG $ is a $d_P$-local minimizer of $P_f$ in the set ($n=k+\n0$)
$$
\cB_k(\cH_r) \igdef 
\{\cG=\{g_i\}_{i\in \IN{k}}\in \hil_r^k: \ \|g_i\|^2=\alpha_i\, , \ i\in \IN{k}\} \ ,
$$ 
because $\la(S_0 + S_\cG ) = (\la\coma \la(S_\cG)\,)^\downarrow \implies 
P_f (\cF_0\coma \cG) = P_f(\cF_0)+P_f(\cG)$ for every $\cG\in \cB_k(\cH_r)$. 
By Lemma \ref{sin F0}, we deduce that $S_1 \in \cG l (\cH_r)^+$, contradicting  
Eq. \eqref{el r de mu}. 
\QED

\section{On the structure of global minimizers of $P_f$ on $\cC_\ca(\cF_0)$}\label{sec5}

\ID \ and let $f:[0,\infty)\rightarrow \R$ be a strictly convex function. 
In this section we obtain a description of the geometrical structure 
of global minimizers of $P_f$ on $\cC_\ca(\cF_0)$. We accomplish this by 
studying the structure of $d_P$-local minimizers of $P_f$ in terms of 
perturbation results for the classical frame design problem. This geometrical 
structure of global minimizers allow us to obtain an finite step algorithm 
that produces a finite set (that does not depend on $f$) which completely 
describes the optimal frame completions $\cF\in\cC_\ca(\cF_0)$ for $P_f$.

\subsection{Partitions into irreducible subsequences} 

From now on we shall fix a strictly convex function  $f:[0 \coma \infty)\rightarrow [0 \coma \infty)$.

\pausa
The goal of  this section is the following  Theorem  on the spectral and geometrical 
structure of global minimizers of $P_f(\cdot)$ on $\cC_\ca^{\rm op}(\cF_0)$. 
The proof is divided into some lemmas that we state after the main result. 
Recall that $\Gamma_d(\cb)= \{ \mu\in (\R_{\geq 0}^d)^\uparrow: \ \cb\prec \mu\}$.

\begin{teo}\label{vale la conjetur}
\ID. 
Denote by  $\lambda=\lambda(S_{\cF_0})$. 
Then 
\ben
\item There exists a vector $\mu = \mu(\lambda \coma \ca \coma f) \in 
\Gamma_d(\cb)$ such that 
\een
\bce 
$\cF=(\cF_0 \coma \cG)\in \cC_\ca(\cF_0)$ is a global minimizer of $P_f \iff
\cF\in \cC_\ca^{\rm op}(\cF_0)$ and $\lambda^\uparrow(S_{\cG})=\mu$ .
\ece
Assume now that $\cF=(\cF_0\coma \cG)$ is a global minimizer of $P_f$ on $\cC_\ca^{\rm op}(\cF_0)$. 
Let $\{\cG_i\}_{i\in \IN{p}}$ be a partition of $\cG$ into irreducible subfamilies, where 
$\cG_i=\{f_j\}_{j\in J_i}\,$ for a partition $\{J_i\}_{i\in \IN{p}}$ of the set of indexes $\{i:\ 1\leq i\leq k\}$. 
Then for each  $i\in \IN{p}\,$ 
\ben
\item[2.]  
The frame operators  $S_{\cG_i}$ and $S_{\cF_0}$ commute. 
\item[3.] 
There exists $c_i\in \R_{> 0}\,$  such that  
$ 
S_\cF \, f_j=c_i\, f_j \peso{for every} 
j\in J_i \ .
$
\end{enumerate}
\end{teo}
\begin{proof}
Item 1 was shown in Theorem \ref{teo sobre unico espectro}. 

\pausa
2. Assume now that $\cF=(\cF_0 \coma \cG )$ is a global minimizer of $P_f$ on $\cC_\ca^{\rm op}(\cF_0)$. Then  
$$ 
S_\cG=\bigoplus_{i\in \IN{p}}S_{\cG_i} \implies \sigma(S_1)=\bigcup_{i\in\IN{p}}\ \sigma(S_{\cG_i})\ .
$$ 
Let $P(\alpha)$ (resp. $P_i(\alpha)$)  denote the spectral projection of $S_\cG$ (resp. $S_{\cG_i}$) associated with $\alpha\in\sigma(S_\cG)$ 
(or 0 in case $\alpha\notin\sigma(S_{\cG_i})$). Then, 
for every $i\in \IN{p}$ we have that 
$$
S_{\cG_i}=\sum_{\alpha\in\sigma(S_\cF)} \alpha\ P_i(\alpha) \peso{with} 
\sum_{i\in \IN{p}} P_i(\alpha)=P(\alpha) \ , \ \ \alpha\in\sigma(S_\cG)\ .
$$ 
Thus, each $P_i(\alpha)$ is a sub-projection of $P(\alpha)$ for $i\in \IN{p}\,$. 
If we consider $\alpha\in\sigma(S_\cG)$, $\alpha\neq 0$, then Corollary 
\ref{loqueimporta} shows that $P_i(\alpha)$ commutes with $S_{\cF_0}\,$, for every $i\in\IN{p}\,$. 
This last fact implies that $S_{\cG_i}$ commutes with $S_{\cF_0}\,$, for every $i\in\IN{p}\,$.

\pausa
3. It is a consequence of item 2  and the following Remark and Lemmas. 
\end{proof}

\begin{rem}\label{miraquienveiene}
\ID\  and let $t=\tr\,\ca$. 
Denote by $S_0 = S_{\cF_0}$ and $\la = \la (S_0)$. Consider the set
$$ 
U_t(S_0\coma m)=\{S_0+B:\   B\in \matpos \, , \ \rk \, B 
\le d-m \ , \ \tr \,(S_0+B)\ =\  t\ \} \ ,
$$
where $n = k+\n0$ and $m = d-k$. It is shown in \cite[Theorem 3.12]{MRS4} that there exist $\prec$-minimizers in $U_t(S_{0}\coma m)$.  
Indeed, there exists $\nu=\nu(\la \coma m)\in (\R^d_{\geq 0})^\downarrow$ such that $S\in U_t(S_{0}\coma m)$ is 
a $\prec$-minimizer if and only if $\la(S)=\nu$. In this case, there exist $c>0$ and  
$\{v_i:\ i\in \IN{d}\}$, an ONB 
for $S_0$ and $\la $ such that 
\begin{enumerate}
\item $S-S_0=\sum_{i=1}^d \rho_i\cdot v_i\otimes v_i$, 
where $\rho=\rho(\lambda\coma m)=\lambda(S-S_0)^\uparrow$;
\item $\nu=(\lambda+\rho^\uparrow)^\downarrow$ and $\lambda_i(S_{0})+\rho_i=c$ whenever $\rho_i\neq 0$.\end{enumerate}
 As a consequence of these facts we get $Sf=c\,f$ for every $f\in R(S-S_0)$. 
Moreover,
if $S\,'\in U_t(S_{0}\coma m)$ is another matrix such that $\lambda(S\,'-S_0)^\uparrow=\rho$ 
and $S\,'-S_0=\sum_{i=1}^d \rho_i\,w_i\otimes w_i\,$, where $\{w_i:\ i\in \IN{d}\}$  
is some  ONB for $S_0$ and $\la$,  then 
$\lambda(S\,')=\nu$ and $S\,'$ is a $\prec$-minimizer in $U_t(S_{0}\coma m)$.

\pausa
Assume now that $\cF=(\cF_0 \coma \cG )\in \cC_\ca(\cF_0)$ is such that $S_{0}$ 
and $S_\cG$ commute.
Denote by 
$$
\cR=R(S_{\cG })  \ \ ,  \ \ \mu= \la\ua(S_{\cG })  \ \ ,  \ \ k' = \rk \,S_{\cG }
\ \ ,  \ \ m' = d-k' = \max \{i \in \IN{d}: \mu_i =0\} 
$$ 
and $\tau=\tr \,\cb $. 
Note that $\cR$ reduces $S_{\cF_0}\,$. Write $S_\cR = S_{\cF_0}|_{\cR} \in  \poRta \,$. 
We get the identity 
\beq\label{Rredux}
S_{\cF_0}+ \poRt = S_{\cF_0}|_{\cR^\perp}\oplus \big(S_{\cR} + \poRta \big) \ ,
\eeq
where $\poRt$ is the sapce defined in Eq. \eqref{poRt}. 
If we identify $\cR$ with $\C^{k'}$ we have  that 
$$
S_\cR + \poRta = U_s(S_\cR \coma 0)\inc \mathcal M_{k'}(\C)\ ,  
$$ 
where $s= \tau + \tr \,S_\cR \,$.
By the previous comments there exists 
$\stau \in S_\cR + \poRta $ 
such that $\la (\stau) = \nu(\la(S_\cR) \coma 0)\in \R_{\geq 0}^{k'}\,$, 
which is a $\prec$-minimizer in $U_s(S_\cR   \coma 0)=S_\cR    + \poRta\, $. 
As a consequence of Eq. \eqref{Rredux} and Remark \ref{rem doublystohasctic}, 
we conclude that 
\bce 
$S_1 \igdef 
S_{\cF_0}|_{\cR^\perp}\oplus \stau\in S_{\cF_0}+ \poRt$ is a $\prec$-minimizer 
in $S_{\cF_0}+ \poRt$ .
\ece 
Notice that $\la(S_1)=\big(\,\lambda(S_{\cF_0}|_{\cR^\perp})\coma \lambda(\stau)\,\big)\da
\in \R^d$. Moreover, by items 1 and 2 above, we see that in this case there exists an 
ONB (for $\cR$) $\{v_i\}_{i\in \IN{k'}}$ for $S_\cR   $ and $\lambda(S_\cR   )\in \R^{k'}$ 
such that 
\beq \label{ecua en R}
\stau-S_\cR   
=\sum_{i\in \IN{k'}}\rho_i\,v_i\otimes v_i \ ,\ \text{ where } \ \ 
\rho =\lambda(\stau -S_\cR   )^\uparrow\in \R^{k'} \ ,
\eeq 
and there exists $c\in \R_{> 0}\,$ such that $\lambda_i(S_\cR   )+\rho_i=c$ whenever $\rho_i\neq 0$.
Hence, in this case we obtain that 
\beq\label{son autovectores en el rango}
S_1 f=c\ f \peso{for every} f\in R(\stau-S_ \cR)\inc \mathcal R\ .
\eeq
\EOE
\end{rem}

\begin{lem}\label{prolem no se}
\ID. Fix a subspace $\cR \inc \C^d$ which reduces $S_{\cF_0}\,$. 
Let $\cF=(\cF_0\coma \cG )\in \cC_\ca(\cF_0)$ be a $d_P$-local minimizer of 
$P_f$ on the set 
$$
\big\{\, \cF'=(\cF_0\coma \cG')\in  \cC_\ca(\cF_0) \, : \,  R(S_{\cG'} ) \inc \cR\, \big\}\ .
$$ 
Assume further that $S_0= S_{\cF_0}$ and $S_{\cG }$ commute and that the sequence $\cG $ 
is {\bf irreducible}. Then 
\begin{enumerate}
\item The frame operator $S_{\cF}$ is a $\prec$-minimizer in $S_{\cF_0}+ \poRt\,$.
\item The subspace $\cR$ is contained in a eigenspace of $S_\cF\,$.
\end{enumerate}
In particular, there exists $c\in \R_{> 0}$ such that $S_{\cF}\,f_i=c\,f_i\,$, 
for $\n0+1\leq i\leq n$.
\end{lem}
\begin{proof}  
Let $k'=\rk \,S_\cG$ 
and $m'=d-k'$.  Since by hypothesis $S_0$ and $S_\cG$ commute, 
arguing as in Remark \ref{miraquienveiene} we conclude that there exists 
$S_1=S_{\cF_0}|_{\cR^\perp}\oplus \stau\in S_{\cF_0}+ \poRt$ such that $S_1$ is a 
$\prec$-minimizer in $S_{\cF_0}+ \poRt\,$.
On the other hand, $S_\cG$ and $S_\cF$ also commute so that there exists an 
ONB of $\C^d$ of eigenvectors of $S_\cF$ and $S_\cG\,$, denoted $\{v_i:\ i\in \IN{d}\}$, such that $\{v_i\}_{i\in \IN{k'}}$ is an ONB for $S_\cR \igdef S_0|_\cR\in 
\poRta  $ and $\lambda(S_\cR   )$. In other words
$$
S_\cF=\sum_{i\in\IN{d}}\alpha_i\cdot v_i\otimes v_i\ \ , \ \   S_\cR   
=\sum_{i\in \IN{k'}}\la_i(S_\cR   )\cdot v_i \otimes v_i\ \ \text{ and } \ \ S_\cG
=\sum_{i\in \IN{k'}}\beta_i\cdot v_i\otimes v_i \  ,
$$ 
 for some $(\alpha_i)_{i=1}^d \in \R_{\geq 0}^d$ and $(\beta_i)_{i=1}^{k'} \in \R_{\geq 0}^{k'}$.
Let $\rho =\lambda(\stau -S_\cR   )\ua\in \R^{k'}$ be as 
in Eq. \eqref{ecua en R}
and consider the continuous curve 
$\s : [0,1] \to S_{\cF_0} +\poRt$ given by 
$$ 
\mathfrak{s}(x)=
S_{\cF_0} + \sum_{i\in \IN{k'}}[x\cdot \beta_i+(1-x)
\cdot \rho_i] \cdot v_i\otimes v_i  \peso{for} x\in [0,1]\ .
$$
First, notice that $\mathfrak{s}(x)$ is a segment (so, in particular, a continuous curve) 
joining $\mathfrak{s}(0)=S_1= S_{\cF_0}|_{\cR^\perp}\oplus \stau$ and 
$\mathfrak{s}(1)=S_\cF\,$. 
Consider now  the map $h : [0,1]\to \R$ given by 
$$
\barr{rl}
h(x)& =\tr \, f\big(\,\mathfrak{s}(x)\,\big) = 
\sum_{i\in \IN{d}}f(\lambda_i(\mathfrak{s}(x))) \\&\\
&= \sum_{i=k'+1}^d f(\al_i) + 
\sum_{i\in \IN{k'}}f \big(\lambda_i(S_\cR   )+x\cdot \beta_i+(1-x)
\cdot \rho_i\,\big) 
\earr
$$ 
for every $x\in [0,1]$.  
Since the sequence  
$\cG $ is irreducible then 
 Corollary \ref{coro sec locales dp},  implies that 
the map $\cS:\cC_\ca(\cF_0)\rightarrow S_{0}+\poRt$ 
defined in Eq. \eqref{La S} has 
a $d_P$-continuous local cross section $\varphi$ 
around $ S_\cF$ such that $\varphi(S_\cF)=  \cF$.
Then, the fact that  $\cF$ is a $d_P$-local minimizer of $P_f$ implies that 
$h$ has a local minimizer at $1\in[0,1]$.   
But this $h$ is a strictly convex function on $[0,1]$ that has a global 
minimum at $x=0$, since $\mathfrak{s}(0)$ is a $\prec$-minimizer in $S_{\cF_0}+\poRt\,$.

\pausa
This implies that $h$ is constant on $[0,1]$ and hence the segment 
$\lambda(\mathfrak{s}(x))$, $x\in [0,1]$, reduces to a point. 
Thus $\beta_i=\rho_i$ for every $i\in \IN{k'}\,$. Hence $S_\cG=\stau-S_\cR   $ 
and $S_\cF=S_{\cF_0}|_{\cR^\perp}\oplus \stau = S_1\,$. 
By Eq. \eqref{son autovectores en el rango} of 
Remark \ref{miraquienveiene},   there exists a $c\in \R_{\geq 0}$ such that 
$S_\cF\,f_i=\stau\,f_i=c\,f_i\,$ for $\n0+1\leq i\leq n$ (since $f_i\in \cR=
R(S_\cG)=R(\stau-S_\cR   )$ for these indexes). This last fact proves item 2 of the statement. 
\end{proof}

\begin{lem}\label{sihayhipvale}
\ID. Let $\cF=(\cF_0 \coma \cG)\in \cC_\ca(\cF_0)$ be a $d_P$-local 
minimizer of $P_f$ on $\cC_\ca(\cF_0)$. 
Let $\{\cG_i\}_{i\in \IN{p}}$ be a partition of $\cG$ into irreducible subfamilies, where 
$\cG_i=\{f_j\}_{j\in J_i}\,$ for a partition $\{J_i\}_{i\in \IN{p}}$ of the set of indexes $\{i:\ \n0+1\leq i\leq n\}$. 
Assume that 
$S_{\cG_i}$ and $S_0$ commute, for every $i\in \IN{p}\,$. 
Then there exist positive numbers 
$$ 
c_1\coma \dots \coma c_p \in  \R_{> 0} \peso{such that} 
S_\cF f_j=c_i\,f_j \ , \ \ j\in J_i \ , \ i\in \IN{p} \ .$$
\end{lem}
\proof
Notice that, by construction, the ranges of the frame operators $S_{\cG_i}$ and $S_{\cG_j}$ are orthogonal whenever $i\neq j$.
Fix $i\in \IN{p}\, $. The hypothesis allows us to apply Lemma \ref{prolem no se} to the 
sequence $(\cF_0 \coma \cG_i)\in \cC_{\cb_i}(\cF_0)$, 
where $\cb_i=(\|f_j\|^2)_{j\in J_i}\,$. In this case we conclude  
that there exists $c_i\in \R_{>0}$ such that $(S_{\cF_0}+S_{\cG_i})\,f_j
=c_i\, f_j\,$, for every $j\in J_i\,$. Hence, 
$$
S_{\cF}\, f_j= (S_{\cF_0}+S_{\cG})\, f_j
=(S_{\cF_0}+\bigoplus_{\ l\in \IN{p}}S_{\cG_l})\,f_j
=(S_{\cF_0}+S_{\cG_i})\,f_j=c_i\, f_j \ , 
$$
for every $j\in J_i\,$. \QED

\subsection{A finite step algorithm to compute global minimizers} \label{sec consec conjetur}

In this section we obtain, as a consequence of Theorem \ref{vale la conjetur}, an algorithmic solution of the optimal frame completion problem with prescribed norms with respect to a general convex potential $P_f$. The key step is the introduction of the following finite set:

\begin{num}
\label{el efin}\rm 
\ID. In order to find the minimizers for the CP with parameters 
$(\cF_0 \coma \ca)$ we construct a finite set $\efin \inc(\R_{\geq  0}^d)^\uparrow $ as follows:

\pausa 
Set $r\in \IN{d}\,$. Consider a partition $\{K_i\}_{i\in \IN{p}}$ of the set 
$\{r,\ldots,d\}$ 
for some $1\leq p\leq (d-r)+1$  and define the subsequences 
of $\la = \la(S_{\cF_0})\,$ given by 
$$
\Lambda_i=\{ \lambda_j\}_{j\in K_i}\in \R_{\geq 0}^{|K_i|}\ ,  \peso{for every}  i\in \IN{p}\ .
$$
Consider also a partition $\{ J_i\}_{i\in \IN{p}}$ of the set $\{1,\ldots,k\}$ and define 
the subsequences of $\cb=(\alpha_i)_{i=1}^k\in \R^k$ given by 
$$
\cb_i=\{\alpha_j\}_{j\in J_i} \in \R_{>0 }^{|J_i|}\ ,  \peso{for every}  i\in \IN{p}\ .
$$
For each $i\in \IN{p}$  define $c_i=|K_i|^{-1} \cdot (\tr \, \Lambda_i +\tr  \, \cb_i)$ 
and $\Gamma_i=\{c_i-\lambda_j\}_{j\in K_i}$. 
Let 
\beq\label{mu stored}
\mu \in \R^d \peso{be given by} 
\mu_j= (\Gamma_i)_j  = c_i-\lambda_j \peso{if}  j \in K_i\ , 
\eeq
and $\mu_j=0$ if $j<r$. 
We now check whether for every $i\in \IN{p}$ it holds that:
\beq\label{tigres}
\Gamma_i\in \R_{\geq 0}^{|K_i|}\ \  , \quad \cb_i\prec\Gamma_i \peso{and that}
 \mu= \mu\ua \in(\R_{\geq 0}^d)^\uparrow \ .
\eeq
In this case 
we declare this  $\mu$ as a member of $\efin  $. Otherwise we drop this $\mu$.  
The set $\efin  $ is then obtained by this procedure, as we vary 
$1\leq r\leq d$ and the partitions previously considered. Therefore, 
$\efin  $ is a finite set.

\pausa
A straightforward computation using Proposition \ref{frame mayo} 
and Eq. \eqref{tigres} shows that
for every $\gamma \in \efin  $ 
there exists a completion  $\cF\,'=(\cF_0 \coma \cG\,')\in \cC_\ca^{\rm op}(\cF_0)$ 
such that $\la\ua(S_{\cG\,'}) = \gamma$ and $\la (S_{\cF\,'})
 = (\la +\gamma)\da$.
We remark that the set $\efin$  can be explicitly 
computed in a finite step algorithm, in terms of $\lambda=\lambda(S_{\cF_0})$ and $\ca$ 
(see Section \ref{bla bla} below for details). \EOE
\end{num}

\pausa
Fix now a strictly convex function $f:[0 \coma \infty)\rightarrow[0 \coma \infty)$. Recall that we denote by 
$F: \R_{\geq 0}^d \to \R_{\geq 0} $ the map given by $F(\gamma) = \sum _{i \in \IN{d}} f(\gamma_i)$ for every 
$\gamma \in \R_{\geq 0}^d\,$.

\begin{teo}\label{consec conjetur}  
\ID \ and let $\la=\la(S_{\cF_0})$.  Then 
\ben 
\item The vector $\mu = \mu(\lambda \coma \ca \coma f) \in (\R_{\ge0}^d)\ua\,$ of 
Theorem \ref{vale la conjetur} satisfies that $\mu \in \efin$.
\item Moreover,  this vector $\mu$ is uniquely determined by the equation 
\beq\label{el mu op}
F(\la + \mu ) = \min \ \{ F(\la+\gamma) : \gamma \in \efin   \ \} \ . 
\eeq
\een
That is, 
a completion $\cF=(\cF_0 \coma \cG)\in \cC_\ca(\cF_0)$ is a $P_f$ global 
minimizer if and only if $\cF\in \caop$,  
$\mu = \la^\uparrow(S_{\cG})\in \efin $ and it  satisfies Eq. \eqref{el mu op}.
\end{teo}
\begin{proof}
Denote by $\mu = \mu(\lambda \coma \ca \coma f) \in (\R_{\ge0}^d)\ua\,$, the vector of Theorem \ref{vale la conjetur}.
Let 
$\cF=(\cF_0 \coma \cG)$ be a global minimizer of $P_f$ on $\cC_\ca(\cF_0)$. In this case, 
by Theorem \ref{dale que va}, $S_{\cF_0}$ and $S_{\cG}$ commute 
and $\lambda(S_\cF)=(\lambda+\mu)^\downarrow$, where $\mu=\mu^\uparrow\in \R_{\geq 0}^d$ 
is such that $\lambda(S_{\cG})=\mu^\downarrow$.

\pausa
Let $\{\cG_i\}_{i\in \IN{p}}$ be a partition of $\cG$ into irreducible subfamilies, 
corresponding to the partition $\{J_i\}_{i\in \IN{p}}$ of $\{\n0+1,\ldots,n\}$, for some $1\leq p\leq d$. 
Notice that in this case $S_{\cG}=\oplus_{i\in \IN{p}} S_{\cG_i}\,$. 
This last fact shows that there exists a partition $\{K_i\}_{i\in \IN{p}}$ such that 
$\lambda(S_{\cG_i})=(\Gamma_i \coma 0_i)$ where $\Gamma_i=\{\mu_j\}_{j\in K_i}$ and $0_i\in \R^{d-|K_i|}$ 
for every $i \in \IN{p}\,$. Then $\mu=(\oplus_{i\in\IN{p}}\Gamma_i)^\uparrow$.

\pausa
Fix $i\in\IN{p}\,$. Theorem \ref{vale la conjetur} implies that there exists $c_i>0$ such that 
$S_\cF f_j=c_i\,f_j$ for every $j\in J_i$ and $S_{\cG_i}$ and that $S_{\cF_0}$ commute. 
This fact implies that $S_\cF|_{R_i}=c_i\,I_{R_i}$, where $R_i=R(S_{\cG_i})$ and 
$P_{R_i}$ denotes the identity operator on $R_i\,$. Therefore, we conclude that 
$c_i=\lambda_j+\mu_j$ for every $j\in K_i\,$. Hence $\Gamma_i=
(c_i-\lambda_j)_{j\in K_i}$ and 
$$
c_i=|K_i|^{-1}\cdot \sum_{j\in K_i} (\lambda_j+\mu_j)=
|K_i|^{-1}\cdot (\tr \, \Lambda_i +\tr \, \{ \alpha_j\}_{j\in J_i}) \ ,
$$
since $S_{\cG_i}=\sum_{j\in J_i}f_j\otimes f_j\,$. This shows that 
$\tr \, S_{\cG_i} = \sum_{j\in J_i} \|f_j\|^2=\sum_{j\in J_i} \alpha_j\,$. 
Moreover, the previous identity and Proposition \ref{frame mayo} imply that 
$\cb_i\prec\Gamma_i\,$, 
where $\cb_i=\{ \alpha_j\}_{j\in J_i}\,$.  
Hence, we conclude that the vector $\mu $ of Theorem \ref{teo sobre unico espectro} satisfies that 
$\mu\in \efin  $, as defined in \ref{el efin}. 

\pausa
As we mentioned before, for every $\gamma \in \efin  $ 
there exists a completion  $\cF\,'=(\cF_0 \coma \cG\,')\in \cC_\ca^{\rm op}(\cF_0)$ 
such that $\la\ua(S_{\cG\,'}) = \gamma$ and $\la (S_{\cF\,'})
 = (\la +\gamma)\da$. Hence the vector $\mu$ 
satisfies Eq. \eqref{el mu op}. 
The converse implication 
now follows from item 1 and Theorem \ref{teo sobre unico espectro}. 
\end{proof}

\begin{rem}\label{si hay minmayo efin hay minmayo}
Let  $\efin  \inc(\R_{\geq  0}^d)^\uparrow$ be the finite set defined in 
\ref{el efin}  and assume that there exists $\mu\in \efin  $ such that 
$\la + \mu$ is a 
$\prec$-minimizer for the set $\la + \efin$ i.e., such that 
\beq\label{posta}
\la+ \mu\prec \la + \ga  \peso{for every}
\ga\in \efin   \ . 
\eeq
Then, by Theorem \ref{consec conjetur} and the comments in Section \ref{subsec 2.2} 
we see that $\mu$ coincides with $\mu(\lambda \coma \ca \coma f)$, the vector of Theorem 
\ref{vale la conjetur}, for all strictly convex functions  
$f:[0 \coma \infty)\rightarrow[0 \coma \infty)$. 

\pausa
That is, given an arbitrary  strictly convex functions  
$f:[0 \coma \infty)\rightarrow[0 \coma \infty)$ then
a completion  $\cF=(\cF_0 \coma \cG)\in \cC_\ca(\cF_0)$ is a global minimizer of $P_f$ 
in $\cC_\ca(\cF_0)$ if and only if $\la^\uparrow(S_{\cG})=\mu$. 
Moreover, a similar argument shows that in this case 
$$
\lambda(S_{\cF_0})+\mu \in \Lambda_\ca^{\rm op}(\lambda(S_{\cF_0})) \peso{is a 
$\prec$-minimizer in}  \Lambda_\ca^{\rm op}(\lambda(S_{\cF_0}))\ .
$$
Therefore $\mu$ (resp. $\lambda(S_{\cF_0})+\mu$) is an structural (spectral) 
solution to the problem of minimizing $P_f\,$, in the sense that the solution 
does not depend of the particular choice of the strictly convex function $f$. 
  
\pausa  
Such structural solutions exist if we assume that the completion problem is feasible (see Remark \ref{caso feasible}). Numerical examples suggest that such a majorization 
minimizer always exists (see Section \ref{bla bla}). These facts induce the following conjecture:
\EOE
\end{rem}

\begin{conj}\label{conj:hay min}\rm \ID. Then there exists $\mu \in \efin$ such that 
$\la_{\cF_0} + \mu $ satisfies the majorization minimality of Eq. \eqref{posta}.  \QED
\end{conj}

\subsection{Algorithmic implementation: some examples and conjectures.}\label{bla bla}
As it was described in the previous section, an algorithm can be developed in order to compute explicitly the set $\efin  $ and the finite set of possible minimizers $\nu=\la+\mu$, $\mu\in \efin  $ constructed from it. A proposed algorithm scheme is the following:

\pausa
\begin{num}\label{algo} \rm
Given the initial data $\la\in (\R_{\geq 0}^d)^\downarrow$ and $\cb=(\alpha_i)_{i=1}^k\,$, 
we set $n=k+\n0$  as before.

\begin{description}
\item[Step 1.]
For each $r\in \IN{r}$ set $\la(r)=(\la_j)_{j=r}^d\,$. 
For such tail of $\la$, of length $l=d-r+1$, we 
consider the minimum $m=\{l,\; k\}$. Now, for each $p\in \IN{m}\,$, 
\bit
\item We compute all possible partitions of $\la(r)$ in $p$ parts. 
We do the same with $\cb$.
\item Fixed a partition for $\la(r)$ and one of $\cb$, we pair the sets 
of both partitions and compute for every pair the constant $c$ 
and check majorization as it was described in Eq. \eqref{tigres}.
\item In case that the majorization conditions are satisfied for all 
pairs in these partitions for $\la(r)$ and $\cb$, 
the vector $\mu$ is constructed 
as in Eq. \eqref{mu stored}. 
\item If $\mu = \mu\ua$ then is $\mu$ stored in the set $\efin$.  
\end{itemize}
\item[Step 2.] The set $\nfin =\{\la+\mu :  \mu\in \efin  \}$ is constructed from that stored data.
\item[Step 3.] We search for the vector   $\nu \in \nfin$ of minimum euclidean norm.
\end{description}
Then this $\nu$ is a minimizer for the map $F(x) = \sum _{i \in \IN{d}} x_i^2$ associated 
to the frame potential 
on the set $\{\la(S_\cF): \cF\in  \cC_\ca(\cF_0)\}$. Moreover   
$\mu = \nu - \la$ is the vector of Theorem \ref{teo sobre unico espectro}, which 
allows to construct (via the Schur-Horn algorithm) optimal 
completions in $\cC_\ca^{\rm op}(\cF_0)$ with respect to the Benedetto-Fickus's frame potential. By Theorem \ref{consec conjetur}, the globarl minimizers corresponding to a different potential in $\cC_\ca^{\rm op}(\cF_0)$ can be computed similarly, i.e. by minimizing  the corresponding convex function on  the set $\nfin$. 

\pausa
{\bf Step 4.} 
Finally, we  test if the vector $\nu$ obtained in  Step 3 is a minimizer for majorization 
in $\nfin$. In that case, the algorithm succeed in finding the minimizer 
for every convex potential $P_f$. \EOE
\end{num}

\pausa
In all examples in which we have applied the previous algorithm, the Step 4  confirmed that the minimizer for the frame potential in $\nfin$ is actually the minimizer for majorization, which suggests  a positive answer to the Conjecture \ref{conj:hay min} (see the comments in Remark \ref{si hay minmayo efin hay minmayo}).

\begin{exa}\label{ejemp Acha}
Consider the set of vectors $\cF_0\in \RS(7\coma 5)$ given in \eqref{ejemF0} and let $\cb=\{3.5\coma 
2\}$ as it was pointed out in Subsection \ref{problemon}, with that initial data, the completion problem is not feasible. Nevertheless, if we apply the algorithm described above, the optimal spectrum $\mu$ and $\nu$ can be computed, since we can describe the set $\nfin$.

\pausa
Indeed in this case $\nfin=\{(9\coma 5\coma 4.5\coma  4\coma  4)\,,\; (9\coma  6.5\coma  5\coma  4\coma  2)\}$ so 
$\nu=(9\coma 5\coma 4.5\coma 4\coma  4)$ (where $\mu=(0\coma 0\coma 0\coma 2\coma 3.5))$ and an optimal completion is given by:
\beq
T_{\cF_1}^* = \mbox{ \scriptsize $  
\left[
\begin{array}{rr}
        0.0441&   -1.3541\\
        0.6901&    0.5701\\
       -1.2093&    0.0887\\
       -0.0569&    0.8836\\
        0.2371&   -0.7435\\
\end{array}
\right]$}
\ . 
\eeq
In this case, the vector $\mu$ is constructed with the partitions $K_1=\{2\}$, $K_2\{1\}$ of the two smaller eigenvalues in $\la=\la(S_{\cF_0})=(9\coma  5\coma  4\coma  2\coma  1) $ which are paired with $J_1=\{2\}$ and $J_2=\{3.5\}$ of $\ca$, using the notation introduced in Section \ref{sec consec conjetur}.

\pausa 
If we now set $\cb=(2\coma  \frac{1}{4}\coma  \frac{1}{4}\coma  \frac{1}{4})$, again the 
problem is not feasible (see \cite{MRS4}). In this case the algorithm yields a 
$\nfin$ with 23 elements with a minimizer for majorization given by $\nu=(9\coma  5\coma  4\coma  3\coma  2.75)$.  In this case, the partitions of $\la$ are $K_1$ and $K_2$ of previous example, and $J_1=\{\frac{1}{4}\coma  \frac{1}{4}\coma  \frac{1}{4}\}$ and $J_2=\{2\}$ is the partition of $\ca$. Finally,  
an optimal  completion of $\cF_0$ with prescribed norms is given by:
\beq
T_{\cF_1}^* = \mbox{ \scriptsize  $
\left[
 \begin{array}{rrrr}
       0.0156&    0.0156&    0.0156&   -1.0236\\
       0.2440&    0.2440&    0.2440&    0.4310\\
      -0.4275&   -0.4275&   -0.4275&    0.0670\\
      -0.0201&   -0.0201&   -0.0201&    0.6679\\
       0.0838&    0.0838&    0.0838&   -0.5620\\ 
\end{array}
 \right]$}
 \ . 
 \eeq
\end{exa}

\begin{exa}  
If $\cb=( 5.35\coma   4.66\coma   3.2\coma   2.5\coma    1.2\coma  1\coma    0.65)$
 and let $\cF_0$ be any family in $\RS(\n0\coma 6) $ such that $\la=\la(S_{\cF_0})=( 5.75\coma    5.4\coma   4.25\coma    4.25\coma    3\coma   2)$, (this is also a non-feasible example) then $\nfin$ has  744 elements, and a minimizer is 
$\nu=(7.505\coma 7.505\coma 7.45\coma 6.9167\coma 6.9167\coma 6.9167)$.
In this example, the partitions for $\la$ ($r_0=1$) and $\ca$ involved in the computation of the optimal $\mu$ are $K_1=\{5.75 \coma 5.4 \coma 4.25\}$, $K_2=\{4.25\}$ and $K_3=\{3\coma 2\}$ and $J_1=\{2.5\coma 1.2\coma 1\coma 0.65\}$, $J_2=\{3.2\}$ and $J_3=\{5.35\coma 4.66\}$ respectively. 
 \EOE
\end{exa}

\pausa
It is worth to note that the number of iterations done in Step 1 grows 
rapidly with $d$ and $k$, and the size of $\nfin$ also increases. As a consequence of these facts,  the  algorithm described in \ref{algo}
is hard to implement for completion problems involving a large number of prescribed norms or for completion problems in $\C^d$ for large $d$.
Nevertheless, in the previous examples (and several others considered for this work) it turned out (besides the fact that Conjecture  \ref{conj:hay min} is verified in all examples) that the index-partition of $\la$ and $\cb$ in the $\prec$-minimizer consist of sets of {\bf consecutive} elements, both for $\la$ and $\cb$. Moreover, in all examples the partitions are paired in such a way that the partitions with the greater elements of $\la$ corresponds to those of $\ca$ with the smaller entries (see the description of $\Lambda_i$ and $J_i$ in previous examples). Moreover, in all examples considered, the minimizer has the property that the sets of vectors corresponding to the partitions with the greater norms of $\ca$ are linearly independent, with the exception of the last partition of $\ca$. This structure is consistent with the  solution for the classical completion problem with  $\cF_0=\emptyset$ ( see \cite{BF,Phys,MR}).

\pausa
This allows to develop a faster algorithm which tests a smaller set of partitions for $\la$ and $\cb$ which reduces considerably the time of computation and data storage. 
Thus, our numerical computations lead to the following Conjecture for the construction of the $\prec$-minimizer:

\begin{conj} \label{conjeturando}
\ID , and assume that $\cb$ is arranged in decreasing order. 
Then, using the notations of \ref{el efin}, the  minimizing vector $\mu \in \efin$ of Theorem \ref{consec conjetur} satisfies that: 
\ben 
\item It is constructed from consecutive partitions of $\la$ and $\cb$. In other words, 
that each set $J_j$ and $K_j$  
in the partitions  $\{J_j\}_{j\in \IN{p}}$ and $\{K_j\}_{j\in \IN{p}}$ given in  \ref{el efin} 
describing $\mu$, consists of { consecutive} indexes.  
\item The partitions 
of $\la$ and $\cb$ are  paired in opposite order: the sets in the partition of $\la$ with the 
larger elements  are compared with those sets in the correspondent partition of $\cb$ with the smaller elements. Moreover, the correspondent sets in both partitions  have the same number of elements, except possibly the  sets with the smallest and greatest  entries of $\ca$ and $\la$ respectively. More explicitly,  there exists $1\leq r_0\leq d$ such that $m=d-r_0+1\leq k$ and 
a sequence $r_0\leq r_1< \ldots< r_{p}= d$ such that:
\[
K_j=\{r_{j-1}+1\coma \ldots \coma  r_j\} \ ,\quad 
J_{j}=\{d-r_j+1\coma  \ldots \coma d-r_{j-1}\} \ ,\peso{for} 2\leq j\leq p\  , 
\]
\[
 K_1=\{r_0 \coma \ldots \coma  r_1\} \ , \ J_1=\{d-r_1+1\coma \ldots\coma  k\} \ ,
\] 
and such that $\mu$ is constructed as in \ref{el efin} in terms of $\{K_j\}_{j\in \IN{p}}$ 
and $\{J_j\}_{j\in \IN{p}}\,$.
\EOE
\een
\end{conj}

\pausa
In the following example we verify that the algorithm implemented following the scheme in \ref{algo} and the simplified (and faster) version of this algorithm that assumes that Conjecture \ref{conjeturando} holds, produce the same solution to the optimal completion problem with respect to the Benedetto-Fickus' frame potential.

\begin{exa}
Given the initial data 
$$
\la=\la(S_{\cF_0})=(7\coma    6\coma    5.5\coma    4\coma    2.5\coma    1\coma    0.5\coma    0.3)
\py \ca=(5\coma   4.5\coma    1.2\coma    1\coma    0.8\coma    0.5)\ ,
$$ 
then applying the algorithm described in \ref{algo} we obtain that the optimal completion 
with prescribed norms $\cF=(\cF_0\coma \cG)$ has eigenvalues 
$\nu=(7\coma 6\coma 5.5\coma 5.3\coma 5\coma 4\coma 3.5\coma 3.5)$. If we only check the 
partitions described in Conjecture \ref{conjeturando}, then we obtain the same optimal 
eigenvalues $\nu$, with the partitions $J_1=\{1.2\coma 1\coma 0.8\coma 0.5\}$, 
$J_2=\{4.5\}$, $J_3=\{5\}$ and $K_1=\{2.5\coma 1 \}$, $K_2=\{0.5\}$, $K_3=\{0.3\}$ 
for $\ca$ and $\la$ respectively ($r_0=5$). But there are only 5 cases constructed 
from this kind of partitions   in a set $\nfin$ with 322 elements. \EOE
\end{exa}

\section{Appendix: Equality in Lindskii's inequality}\label{secAppendix}
Fix $S_0\in \matpos$. In this section we characterize those matrices 
\beq\label{los S1}
S_1 \in \matpos \peso{such that} 
\la(S_0+S_1) = \Big(\, \la\da(S_0)+ \la\ua(S_1)\,\Big)^\downarrow  \ .
\eeq
If $S_1 \in \matpos$ satisfies Eq. \eqref{los S1} then we say that $S_1$ is 
an {\bf optimal matching matrix} for $S_0\,$. Note that optimal matching 
matrices correspond to the cases of equality in Lindskii's inequality, 
as stated in Theorem \ref{LindsTeo}. 
 
\pausa
Although we have defined this notion for positive matrices (since we interested 
in its application to frame operators) similar definitions and conclusions 
holds for general hermitian matrices (by translations by convenient multiples of the identity).

\subsection{Optimal matching matrices commute}\label{lindskii para matrices}

In this section we study the case of equality in Lindskii's inequality and show that if $S_1$
is an optimal matching for $S_0$ (i.e. $S_1$ is as in Eq. \eqref{los S1}) then $S_0\,S_1= S_1\,S_0\,$.

\pausa
We begin by revisiting some classical matrix analysis results. We shall give 
short proofs of them in order to handle these proofs for the equality cases in 
which we are interested here. 
\begin{lem}[Weyl's inequalities] \label{lemWineq}
Let $A,\,B\in \matsad$. Then,
\beq\label{eqW1}
\lambda_j(A+B)\leq \lambda_i(A)+\lambda_{j-i+1}(B) \peso{for} i\leq j\,,
\eeq
\beq\label{eqW2}
\lambda_j(A+B)\geq \lambda_i(A)+\lambda_{j-i+d}(B) \peso{for} i\geq j\,.
\eeq
Moreover, if there exists 
$i\leq j$ (resp. $i\geq j$) such that 
\beq\label{eqW3}\lambda_j(A+B)=\lambda_i(A)+\lambda_{j-i+1}(B)\eeq
(resp. $\lambda_j(A+B)= \lambda_i(A)+\lambda_{j-i+d}\,(B)$) 
then there exists a unit vector $x$ such that $$ (A+B)\,x=\lambda_j(A+B)\,x \, , \  A\,x=\lambda_i(A)\,x \, , \ B\,x=\lambda_{j-i+1}(B)\,x \ ,$$
(resp. $(A+B)\,x=\lambda_j(A+B)\,x \, , \  A\,x=\lambda_i(A)\,x \, , \ B\,x=\lambda_{j-i+d}\,(B)\,x$).
\end{lem}
\begin{proof}
We begin by proving \eqref{eqW1}. Let $u_j \coma v_j$ and $w_j$ denote the 
eigenvectors of $A,\,B$ and $A+B$ respectively, corresponding to their 
eigenvalues arranged in decreasing order. Let $i\leq j$ and consider the three 
subspaces spanned by the sets $\{w_1,\ldots,w_j\}$, $\{u_i,\ldots,u_n\}$ and $\{v_{j-i+1},\ldots,v_n\}$. Since the dimensions of these subspaces are $j$, 
$n-i+1$ and $n-j+i$ respectively, we see that they have a non trivial intersection. 
If $x$ is a unit vector in the intersection of these subspaces then 
$$
\lambda_j(A+B)\leq \langle \,(A+B)\, x\coma x\rangle= \langle A \,x\coma x\rangle
+ \langle B\,x\coma x\rangle\leq \lambda_i(A)+\lambda_{j-i+1}(B)\ .
$$ 
If we further assume that equality \eqref{eqW3} holds for these indexes then we deduce that 
$$ 
\langle \,(A+B)x\coma x\rangle= \lambda_j(A+B)  \ \ \coma \ \ 
\langle A \,x\coma x\rangle=\lambda_i(A) 
\py \langle B\,x\coma x\rangle=\lambda_{j-i+1}(B) \ . 
$$
Because $x$ lies in the intersection of the previous subspaces, these last 
facts imply that $ (A+B)\,x= \lambda_j(A+B)\,x$, $ A \,x=\lambda_i(A) \,x$ and 
$\langle B\,x,x\rangle=\lambda_{j-i+1}(B)\,x$. 
The inequality \eqref{eqW2} and the equality \eqref{eqW3} 
for the case $i\ge j$ follow similarly.
\end{proof}

\begin{cor}[Weyl's monotonicity principle]\label{lemWmp}
Let $A\in\matsad$ and $B\in\matpos$. Then 
\beq\label{eqWM1}
\lambda_j(A+B)\geq \lambda_j(A)\ \peso{for every} j\in \IN{d}\  . 
\eeq 
If there exists $J\inc\IN{d}$ such that $\lambda_j(A+B)= \lambda_j(A)$ for every 
$j\in J$, then there exists an orthonormal system 
$\{x_j\}_{j\in J}$ such that $A\,x_j=\lambda_j(A)\,x_j$ and $B\,x_j=0$
for every $j\in J$.
\end{cor}
\begin{proof} 
Inequality  \eqref{eqWM1} follows easily from 
Lemma \ref{lemWineq} (with $i=j$). 
The second part follows by induction on the set $|J|$: \  
Fix $j_0\in J$. By Eq. \eqref{eqW2} with $i= j = j_0\,$, 
there exists a unit vector $x_{j_0}$ such that $A\,x_{j_0}
=\lambda_{j_0}(A) \,x_{j_0}$ and $B\,x_{j_0}=\lambda_d(B)\,x_{j_0}=0$.

\pausa
This proves the case $|J|=1$. If $|J|>1$, 
consider the space $W=\{x_{j_0}\}^\perp\inc\C^d$ which 
reduces $A$,   $B$  and  $A+B$. Let $I=\{j:\ j\in J\, , \  j<j_0\}
\cup\{j-1:\ j\in J\, , \ j>j_0\}$. The operators 
$A|_W\in L(W)^{\rm sa}$ and $B|_W\in  L(W)^+$ satisfy  that 
$\la_j(A|_W+B|_W)= \la_j(A|_W)$ for every $j\in I$,  
with  $|I| = |J|-1$. By the inductive hypothesis we 
can find an orthonormal system $\{x_j\}_{j\in I} \inc W$ which 
satisfy the desidered properties. 
\end{proof}

\pausa

\begin{pro}\label{proLeq} 
Let $A,\,B\in \matsad$. Then the equality  
$$
\big(\, \lambda(A+B)-\lambda(A)\, \big)^\downarrow=\lambda(B)
\implies \ \ \mbox{$A$ and $B$ commute}\ .
$$ 
\end{pro}
\begin{proof}
We can assume that $B$ is not a multiple of the identity. 
By hypothesis, there exists permutation $\sigma\in \mathbb S_d$ such that  $\lambda_j(B)=\lambda_{\sigma(j)}(A+B)-\lambda_{\sigma(j)}(A)$ 
for every $j\in \IN{d}\,$. 
Therefore, there exists an increasing sequence $\{J_k\}_{k=1}^d$ of subsets of $\IN{d}$ such that $|J_k|=k$ and 
\beq\label{ecuacuacuac}
\sum_{j\in J_k} \lambda_{j}(A+B)-\lambda_{j}(A)=\sum_{j=1}^k \lambda_j(B)
\peso{for every} k\in \IN{d}\ .
\eeq
Let $k\in \IN{d}$ be such that $\lambda_{k-1}(B)>\lambda_k(B)$ 
(recall that $B\neq \alpha\, I$ for $\alpha\in \R$). 
Let us denote by $B_k=B-\lambda_k(B)\,I$ and notice 
Eq. \eqref{ecuacuacuac} also holds if we replace $B$ by $B_k\,$.

\pausa
By construction $\lambda_k(B_k)=0$ and the orthogonal projection 
onto the kernel of the  positive part $B_k^+\in \matpos$ coincides 
with the spectral projection of the $B$ associated to the interval
$(-\infty, \lambda_k(B)]$. Moreover, $\dim\ker B_k^+=d-k+1$.

\pausa
Since $B_k^+\in \matpos$ and $B_k\leq B_k^+$ then Weyl's 
monotonicity principle implies that 
$$ 
\lambda_j(A+B_k)\leq \lambda_j(A+B_k^+) \, , \ 
j\in \IN{d}\implies
 \sum_{j\in J_{k-1}} \lambda_j(A+B_k) \leq \sum_{j\in J_{k-1}} \lambda_j(A+B_k^+)\ .  $$
Therefore 
\begin{eqnarray*}
\sum_{j\in J_{k-1}} \lambda_j(A+B_k) - \lambda_j(A)&\leq &\sum_{j\in J_{k-1}} \lambda_j(A+B_k^+) - \lambda_j(A) \\
&\leq & \sum_{j\in \IN{d}} \lambda_j(A+B_k^+) - \lambda_j(A) \\ &=& \tr \,(A+B_k^+)-\tr \,A =\sum_{j=1}^{k-1}\lambda_j(B_k)
\end{eqnarray*} since $\lambda_j(A+B_k^+) \geq \lambda_j(A)$ for $j\in \IN{d}$ - again by Weyl's monotonicity principle - and since, by hypothesis, $\lambda_k(B_k)=0$. 
The inequalities above are the key part of the proof of 
Lindskii's Theorem \ref{LindsTeo} ($\lambda(A+B)-\lambda(A)\prec \lambda(B)$\,).  But  here they 
actually equalities, by Eq. \eqref{ecuacuacuac}. 

\pausa
Let $J_{k-1}^c=\IN{d}\setminus J_{k-1}$. Then, 
from the above equalities  we get that $\lambda_j(A+B_k^+)=\lambda_j(A)$
for every $j\in J_{k-1}^c\,$. By Corollary \ref{lemWmp} there exists 
an ONS $\{x_j\}_{j\in J_{k-1}^c}$ such that $A\,x_j=\lambda_j(A)\,x_j$ 
and $B_k^+\,x_j=0$ for every $j\in J_{k-1}^c\,$. All these facts together 
imply that 
$$
P_k\igdef \sum_{j\in J_{k-1}^c} x_j\otimes x_j 
= P_{\ker B_k^+} \py P_k \, A = A\, P_k \ .
$$ 
Recall that $P_k$ is also the spectral projection of $B$ associated to the interval $(-\infty, \lambda_k(B)]$, for any $k\in \IN{d}$ such that $\lambda_{k-1}(B)>\lambda_k(B)$. Since the spectral projection of $B$ associated with $(-\infty, \lambda_1(B)]$ equals the identity operator, 
and $B$ is a linear combination of the projections $P_k$ and $I$, 
we conclude that $A$ and $B$ commute.
\end{proof}

\pausa
Now we are ready to prove that if $S_1\in \matpos $ is 
as in Eq. \eqref{los S1} then $S_0\,S_1= S_1\,S_0\,$.

\begin{teo}\label{teoleq}
Let $S_0 \coma S_1\in \matsad$ be such that $\la(S_0+S_1)=\big(\,\la(S_0)+\la\ua(S_1)\,\big)\da$. 
Then $S_0$ and $S_1$ commute.
\end{teo}
\begin{proof}
Take $B=S_0+S_1$ and $A=-S_1\,$. Therefore 
 $-\la(A)=\la\ua(-A)= \la\ua(S_1)$, so that  
$\la(A+B) -\la (A) = \la(S_0)+\la\ua(S_1)$. 
Hence $A$ and $B$ satisfy  the assumptions in Proposition \ref{proLeq}  and they must commute. In this case $S_0$ and $S_1$ also commute.
\end{proof}

\subsection{Characterization of optimal matching matrices}

Let $S_0\in\matpos$ and let $S_1\in\matpos$ be an optimal matching matrix for $S_0$. Then, Theorem \ref{teoleq} implies that $S_0\,S_1= S_1\,S_0\,$ and hence there exists a common
ONB of eigenvectors for $S_0$ and $S_1$. In order to complete describe $S_0$ and $S_1$ we first consider some technical results.

\pausa
We begin by fixing some notations. 
Let $\lambda \in \R_{> 0}^d\,$. For every $j \in \IN{d}$ we define
the set 
$$
L(\la \coma j)=\{i\in \IN{d}:\lambda_i=\lambda_j\} \ .
$$
If we assume that $\lambda=\lambda^\downarrow$ or $\lambda=\lambda^\uparrow$ then the sets $L(j)$ are formed by consecutive integers. In the firs case we have that $\la_i<\la _j \implies 
k>l $ for every $k \in  L(\la \coma i) $ and $l \in L(\la \coma j)$. 

\pausa
Given a permutation $\sigma\in \mathbb S_d$ and $\lambda \in \R_{> 0}^d\,$ 
we denote by $\la_\sigma = (\la_{\sigma(1)}
\coma \dots \coma \la_{\sigma(d)})$. Observe that 
\beq\label{sigma lo fija}
\la = \la_\sigma \iff \la = \la_{\sigma\inv} \iff
\sigma\big(L(\la \coma j)\,\big) = L(\la \coma j)
\peso{for every} j\in \IN{d}\ .
\eeq
The following inequality is well known (see for example \cite[II.5.15]{Bat}):

\begin{pro}[Rearrangement inequality for products of sums]\label{propwk1}
\rm 
Let $\la \coma\mu\in \R_{> 0}^d$ be such that $\la=\la\da$ 
and $\mu=\mu\ua$. Then  
$\prod_{i=1}^d (\lambda_i+\mu_i)\geq \prod_{i=1}^d (\lambda_i+\mu_{\sigma(i)})
$ for every permutation  $\sigma\in \mathbb S_d\,$. 
\end{pro}  

\pausa
The following result deals with the case of equality in the last inequality.

\begin{pro}\label{propwk2 b}
Let $\lambda,\,\mu\in \R_{> 0}^d$ be such that $\lambda=\lambda^\downarrow$
and $\mu=\mu^\uparrow$. Let $\sigma\in \mathbb S_d$ be such that 
$$ 
(\lambda+\mu)^\downarrow = (\lambda+\mu_\sigma)^\downarrow \ .
$$ 
Moreover, assume that $\sigma$ also satisfies that: 
\beq\label{la hyp b}
\mbox{\rm if  \ $r\coma s \in \IN{d}$  \ are such that \  $\mu_{\sigma(r)}=\mu_{\sigma(s)}$  \ with  \ 
$\sigma(r)<\sigma(s)$  \ then  \ $r<s$ .}
\eeq 
Then the permutation $\sigma$ satisfies that  $\lambda=\lambda_{\sigma}\,$. 
\end{pro}
\begin{proof} For every $\tau\in \mathbb S_d$ let 
$F(\tau)=\prod_{i=1}^d (\lambda_i+\mu_{\tau(i)})$. 
By the hypothesis and Proposition \ref{propwk1}, 
$$
F(\sigma)=F({\rm id})=\max_{\tau\in \mathbb S_d}F(\tau)\ .
$$
Assume that $\lambda\neq \lambda_{\sigma^{-1}}\,$. 
In this case there exists $j\coma k\in \IN{d}$ such that 
\beq\label{las 4 b}
\mu_j<\mu_k 
\py  \lambda_{\sigma^{-1}(j)}<\lambda_{\sigma^{-1}(k)} \ . 
\eeq
Indeed, let $j_0$ be the smallest index such that $\sigma^{-1}$ 
does not restrict to a permutation on $L(\la \coma j_0)$. 
Then, there exists $j\in L(\la \coma j_0)$ such that 
$\sigma^{-1}(j)\notin L(\la \coma j_0)$. 
As $\sigma^{-1}(L(\la \coma j_0)\setminus\{j\})\neq L(\la \coma j_0)$ 
there also 
exists $k\notin L(\la \coma j_0)$ such that 
$\sigma^{-1}(k)\in L(\la \coma j_0)$.
They have the required properties: 

\bit
\item 
First note that 
$\lambda_{\sigma^{-1}(j)}<\la_{j_0} = \lambda_{\sigma^{-1}(k)} $ 
(and then also $\sigma^{-1}(j)>\sigma^{-1}(k)\,$) because 
$\sigma^{-1}(j) $ can not be in $L(\la \coma j_0)$ nor in 
$L(\la \coma r)$ for any $r<j_0\,$ (where $\sigma\inv$ acts 
as a permutation).

\item A similar argument shows that $j<k$. We have used in both cases that the sets 
$L(\la \coma j)$ are formed by consecutive integers,  since the vector $\lambda$ is 
decreasingly ordered.  
\item 
Observe that $j<k \implies \mu_j\leq \mu_k\,$.  So it suffices to show 
that  $\mu_j\neq\mu_k\,$. Let us denote by $r=\sigma^{-1}(j) $ and 
$ s=\sigma^{-1}(k)$. The previous items show that $r>s$ and $\sigma(r)<\sigma(s)$. 
Hence the equality  $\mu_j=\mu_{\sigma(r)}=\mu_{\sigma(s)}=\mu_k\,$ 
is forbidden  by our hypothesis \eqref{la hyp b}. 
\eit
So Eq. \eqref{las 4 b} is proved. Consider now the 
permutation $\tau=\sigma^{-1}\circ (j\coma k)$, 
where $(j\coma k)$ stands for the transposition of the indexes $j$ and $k$. 
Straightforward computations show that 
$$
(\lambda_{\sigma^{-1}(j)}+\mu_j)\, (\lambda_{\sigma^{-1}(k)}+\mu_k) - (\lambda_{\sigma^{-1}(j)}+\mu_k) \, (\lambda_{\sigma^{-1}(k)}+\mu_j) = (\lambda_{\sigma^{-1}(j)}-\lambda_{\sigma^{-1}(k)}) \, (\mu_k-\mu_j)
\stackrel{\eqref{las 4 b}}{<} 0 \ .
$$
From the previous inequality we conclude  that $F({\rm id})=F(\sigma)<F(\tau)\leq F({\rm id}) . $ This contradiction arises from the assumption $\lambda\neq \lambda_{\sigma^{-1}}$. Therefore  $\lambda= \lambda_{\sigma^{-1}}\stackrel{\eqref{sigma lo fija}}{=} \la_\sigma$ as desired.
\end{proof}

\begin{rem}\label{sigma prima}
Let $\la \coma \mu\in \R_{> 0}^d$ be such that $\la=\la\da$ and $\mu=\mu\ua$. 
Let $\tau\in \mathbb S_d$ be such that $(\lambda+\mu)^\downarrow 
= (\lambda+\mu_\tau)^\downarrow$. 
Then, by considering convenient permutations of the sets 
$L(\mu \coma j)$ we can always replace $\tau $ by $\sigma$
in such a way that $\mu_\sigma= \mu_{\tau}$ and such that 
this $\sigma$ satisfies the condition \eqref{la hyp b} of 
Proposition \ref{propwk2 b}. 
Hence, in this case $(\lambda+\mu)^\downarrow = (\lambda+\mu_{\sigma})^\downarrow$ 
and the previous result applies.
\EOE
\end{rem}

\begin{teo}[Equality in Lindskii's inequality]\label{LA ONB}
Let $S_0\in\matpos$ and let $S_1\in\matpos$ be an optimal matching matrix for $S_0\,$. 
Let $\la=\la(S_0)$ and $\mu=\la\ua(S_1)$.  
Then there exists 
$\{v_i: i\in \IN{d}\}$ a ONB for $S_0$ and $\la$ such that 
\beq\label{la BON}
S_1 = \sum_{i\in \IN{d}} \, \mu_i \cdot v_i \otimes v_i 
\py 
S_0 + S_1 = \sum_{i\in \IN{d}} \, (\la_i + \mu_i) \, v_i \otimes v_i \ .
\eeq
\end{teo}
\proof
Let us assume further that $S_0 \coma \,S_1$ are invertible matrices so that 
$\la \coma \mu\in \R^d_{>0}\,$. 
By Theorem \ref{teoleq} we see that $S_0$ and $S_1$ commute. Then, there exists 
$\cB= \{w_i: i\in \IN{d}\}$ an ONB for $S_0$ and $\la$ such that $S_1 \, w_i 
= \mu_{\tau(i)}\, w_i$ for every $i\in \IN{d}\,$, and for some 
permutation  $\tau \in \mathbb S_d\,$. Therefore 
$$
\big(\,\la +\mu\,\big)\da 
\stackrel{\eqref{los S1}}{=} \la (S_0+S_1) = \big(\,\la +\mu_\tau \,\big)\da  \ .
$$

\pausa
By Remark \ref{sigma prima} we can replace $\tau$ by $\sigma \in \mathbb S_d$ in such a way that 
$\mu_\tau = \mu_\sigma\,$,  $(\lambda+\mu)^\downarrow = (\lambda+\mu_\sigma)^\downarrow$ and $\sigma $ satisfies the hypothesis \eqref{la hyp b}. Hence, by  
Proposition \ref{propwk2 b}, we deduce that 
$\la_{\sigma\inv} = \la$. 
Therefore one easily 
checks that the ONB formed by the vectors $v_i = w_{\sigma\inv(i)}$ for 
$i \in \IN{d}$ (i.e. the rearrangement $\cB_{\sigma\inv}$ of $\cB$) 
is still a ONB for $S_0$ and $\la$, but it now 
satisfies Eq. \eqref{la BON}. 

\pausa
In case $S_0$ or $S_1$ are not invertible, we can argue as above with the matrices 
$\tilde S_0=S_0+I$ and $\tilde S_1=S_1+I$. These matrices are invertible and such 
that $\tilde S_1$ is an optimal matching for $\tilde S_0\,$. 
Further, $\la(\tilde S_0)=\la(S_0)+ \uno$ and $\la(\tilde S_1)=\la(S_1)+\uno$. Hence, 
if $\{v_i:\ i\in \IN{d}\}$ has the desired properties for $\tilde S_0$ and 
$\tilde S_1$ then this ONB also has the desired properties for $S_0$ and $S_1\,$.
\QED

\fontsize {8}{9}\selectfont

\end{document}